\documentclass[11pt]{amsart}
\usepackage{amsmath}
\usepackage{a4wide}
\usepackage[utf8]{inputenc}
\usepackage{amssymb}
\usepackage{amsopn}
\usepackage{epsfig}
\usepackage{amsfonts}
\usepackage{latexsym}
\usepackage{amsthm}
\usepackage{enumerate}
\usepackage[UKenglish]{babel}
\usepackage{verbatim}
\usepackage{color}
\usepackage{calrsfs}
\usepackage{pgf}
\usepackage{tikz}

\usepackage{mathrsfs}   % This provides an alternative way to typeset script capital letters, using \mathscr{}
 \usetikzlibrary{arrows,automata}

\DeclareMathAlphabet{\pazocal}{OMS}{zplm}{m}{n}
\newtheorem{theorem}{Theorem}[section]
\newtheorem{lemma}[theorem]{Lemma}
\newtheorem{proposition}[theorem]{Proposition}
\newtheorem{corollary}[theorem]{Corollary}
 \newtheorem{main}{Theorem}
  \newtheorem{cmain}[main]{Corollary}

\theoremstyle{definition}
\newtheorem{definition}[theorem]{Definition}
\newtheorem{example}[theorem]{Example}

\newtheorem{innercustomthm}{Theorem}
\newenvironment{alternativetheorem}[1]
  {\renewcommand\theinnercustomthm{#1$'$}\innercustomthm}
  {\endinnercustomthm}

\theoremstyle{remark}
\newtheorem{remark}[theorem]{Remark}

\numberwithin{equation}{section}

\newcommand{\R}{\ensuremath{\mathbb{R}}}
\newcommand{\N}{\ensuremath{\mathbb{N}}}

\renewcommand{\c}{ {\mathbf{c}}}
\renewcommand{\d}{ {\mathbf{d}}}

\newcommand{\bb}{\mathcal{B}}

\renewcommand{\u}{\ensuremath{\pazocal{U}}}
\newcommand{\ub}{\mathcal{U}}
\newcommand{\hub}{{\mathcal{\hat U}}}
 
\newcommand{\us}{\mathbf{U}}

\newcommand{\vs}{ {\mathbf{V}}}

\newcommand{\set}[1]{\left\{#1\right\}}
\newcommand{\la}{\lambda}
\newcommand{\ga}{\gamma}

\newcommand{\ep}{\varepsilon}
\newcommand{\f}{\infty}
\newcommand{\de}{\delta}

\newcommand{\al}{\alpha}
\newcommand{\lle}{\preccurlyeq}
\newcommand{\lge}{\succcurlyeq}

\newcommand{\si}{\sigma}

\newcommand{\ra}{\rightarrow}
\begin{document}

\title{Bifurcation sets arising from non-integer base expansions}

\author{Pieter Allaart}
\address[P. Allaart]{Mathematics Department, University of North Texas, 1155 Union Cir \#311430, Denton, TX 76203-5017, U.S.A.}
\email{allaart@unt.edu}

\author{Simon Baker}
\address[S. Baker]{Mathematics institute, University of Warwick, Coventry, CV4 7AL, UK}
\email{simonbaker412@gmail.com}

\author{Derong Kong}
\address[D. Kong]{Mathematical Institute, University of Leiden, PO Box 9512, 2300 RA Leiden, The Netherlands}
\email[Corresponding author]{d.kong@math.leidenuniv.nl}

\date{\today}
\dedicatory{}

%\begin{frontmatter}


\begin{abstract}
Given a positive integer $M$ and $q\in(1,M+1]$, let $\pazocal U_q$ be the set of $x\in[0, M/(q-1)]$ having a unique $q$-expansion: there exists a unique sequence $(x_i)=x_1x_2\ldots$ with each $x_i\in\set{0,1,\ldots, M}$ such that
 \[
 x=\frac{x_1}{q}+\frac{x_2}{q^2}+\frac{x_3}{q^3}+\cdots.
 \]
 Denote by $\mathbf U_q$ the set of  corresponding sequences of  all points in  $\pazocal U_q$.
 It is well-known that the function $H: q\mapsto h(\mathbf U_q)$ is a Devil's staircase, where $h(\mathbf U_q)$ denotes the topological entropy of $\mathbf U_q$. In this paper we {give several characterizations of} the bifurcation set
 \[
 \mathcal B:=\set{q\in(1,M+1]: H(p)\ne H(q)\textrm{ for any }p\ne q}.
 \]
Note that $\mathcal B$ is contained in the set {$\mathcal{U}$} of bases $q\in(1,M+1]$ such that $1\in\pazocal U_q$. By using {a} transversality technique we also calculate the Hausdorff dimension of the difference {$\mathcal{U}\setminus\mathcal B$}. Interestingly this quantity is always strictly between $0$ and $1$. When $M=1$ the Hausdorff dimension of {$\mathcal{U}\setminus\mathcal B$} is $\frac{\log 2}{3\log \la^*}\approx 0.368699$, where $\la^*$ is the unique root in $(1, 2)$ of the equation $x^5-x^4-x^3-2x^2+x+1=0$. 
\end{abstract}
\maketitle

\section{Introduction}\label{s1}
Fix a positive integer $M$. 
For $q\in(1,M+1]$, a sequence $(x_i)=x_1x_2\ldots$ with each ${x_i}\in\set{0,1,\ldots, M}$ is called a \emph{$q$-expansion} of $x$ if
\begin{equation}\label{e11}
x=\sum_{i=1}^\f\frac{x_i}{q^i}=:\pi_q((x_i)).
\end{equation}
Here the \emph{alphabet} $\set{0,1,\ldots, M}$ will be fixed throughout the paper. 
Clearly, $x$ has a $q$-expansion if and only if $x\in I_q:=[0, M/(q-1)]$. When $q=M+1$ we know that each $x\in I_{M+1}=[0, 1]$ has a unique $(M+1)$-expansion except for countably many points, which  have precisely two expansions. When $q\in(1,M+1)$ the set of expansions of an $x\in I_q$ can be much more complicated. Sidorov showed in \cite{Sidorov_2007} that Lebesgue almost every $x\in I_q$ has a continuum of $q$-expansions. Therefore, the set of $x\in I_q$ with a unique $q$-expansion is negligible in the sense of Lebesgue measure. On the other hand, the {third} author and his coauthors showed in \cite{Kong_Li_Dekking_2010} ({see also Glendinning and Sidorov \cite{Gle-Sid-01} for the case $M=1$})  that the set of $x\in I_q$ with a unique $q$-expansion   has positive Hausdorff dimension when $q>q_{KL}$, where $q_{KL}=q_{KL}(M)$ is the \emph{Komornik-Loreti constant} (see Section \ref{s2} for more details).

For $q\in(1,M+1]$ let $\u_q$ be the  \emph{univoque set} of $x\in I_q$ having a unique $q$-expansion. This means that for any $x\in\u_q$ there exists a unique sequence $(x_i)\in\set{0,1,\ldots, M}^\N$ such that $x=\pi_q((x_i))$. Denote by
$\us_q=\pi_q^{-1}(\u_q)$ the corresponding set of $q$-expansions.
Note that   $\pi_q$ is a bijection from $\us_q$ to $\u_q$.   So the study of the univoque set $\u_q$ is equivalent  to the study of the \emph{symbolic univoque set} $\us_q$.

{De Vries and Komornik \cite{DeVries_Komornik_2008} discovered an intimate connection between $\u_q$ and the set
\begin{equation} \label{eq:U}
\ub:=\set{q\in(1,M+1]: 1\in\u_q}
\end{equation}
of bases for which the number 1 has a unique expansion. For $M=1$, the set $\ub$ was first studied by Erd\H{o}s et al.~\cite{Erdos_Joo_Komornik_1990,Erdos_Horvath_Joo_1991}. They showed that the set $\ub$ is uncountable,  of first category and of zero Lebesgue measure. Later, Dar\'{o}czy and K\'{a}tai \cite{Daroczy_Katai_1993} proved that the set $\ub$ has full Hausdorff dimension.   Komornik and Loreti \cite{Komornik_Loreti_2007} showed that the topological closure $\overline{\ub}$ is a \emph{Cantor set}: a non-empty perfect set with no interior points. Indeed, for general $M\ge 1$, the above properties of $\ub$ also hold (cf.~\cite{Vries-Komornik-Loreti-2016, Komornik_Kong_Li_2015_1}). Some connections with dynamical systems, continued fractions and even the Mandelbrot set can be found in \cite{Bon-Car-Ste-Giu-2013}. 
}

\subsection{Set-valued bifurcation set $\hub$}

Let ${\Omega:=}\set{0, 1, \ldots, M}^\N$ be the set of all sequences with each element from $\set{0, 1,\ldots, M}$. Then $({\Omega}, \rho)$ is a compact metric space with respect to the metric $\rho$ defined by
\begin{equation}\label{e12}
\rho((c_i), (d_i))=(M+1)^{-\inf\set{j\ge 1: c_j\ne d_j}}.
\end{equation}
Under the metric $\rho$ the Hausdorff dimension of any subset  $E\subseteq {\Omega}$ is  well-defined.

Note that the set-valued map $F: q\mapsto \us_q$ is increasing, i.e., $\us_p\subseteq\us_q$ for any $p, q\in(1, M+1]$ with  $p<q$ (see Section \ref{s2} for more explanation). In \cite{DeVries_Komornik_2008} de Vries and Komornik showed that the map $F$ is locally constant almost everywhere. On the other hand,  the third author and his coauthors proved in \cite{Kong_Li_Lv_Vries2016} that there exist infinitely many  $q\in(1,M+1]$ such that the difference between $\us_q$ and $\us_p$ for any $p\ne q$ is significant: $\us_q\bigtriangleup \us_p$ has positive Hausdorff dimension, where $A\bigtriangleup B=(A\setminus B)\cup(B\setminus A)$ stands for the {symmetric difference} of two sets $A$ and $B$. Let $\hub$ be the \emph{bifurcation set} of the set-valued map $F$,  defined by
\[
\hub=\hub(M):=\set{q\in(1,M+1]: \dim_H(\us_p\bigtriangleup\us_q)>0\textrm{ for any }p\ne q}.
\]
{Compared to the set $\ub$ from \eqref{eq:U}, {we know by \cite[Theorems 1.1 and 1.2]{Kong_Li_Lv_Vries2016}}   that $\hub\subset\ub$ and the difference $\ub\backslash\hub$ is countably infinite. As a result, $\hub$ is a Lebesgue null set of full Hausdorff dimension.} 
Furthermore,
\begin{equation}\label{e13}
(1,M+1]\setminus\hub=(1,q_{KL}]\cup\bigcup[q_0, q_0^*].
\end{equation}
The union on the right hand-side of (\ref{e13}) is pairwise disjoint and countable. By the definition of $\hub$ it follows that each connected component $[q_0, q_0^*]$ is a maximum interval such that the difference $\us_{q_0}\bigtriangleup\us_{q_0^*}=\us_{q_0^*}\setminus\us_{q_0}$ has zero Hausdorff dimension. So the closed interval $[q_0, q_0^*]$ is called a \emph{plateau} of $F$. Indeed, for any $q\in(q_0, q_0^*)$ the difference $\us_q\setminus\us_{q_0}$ is at most countable, and for $q=q_0^*$ the difference $\us_{q_0^*}\setminus\us_{q_0}$ is uncountable but of zero Hausdorff dimension (cf.~\cite[Lemma 3.4]{Kong_Li_Lv_Vries2016}). Furthermore, each left endpoint $q_0$ is an algebraic integer, and each right endpoint $q_0^*$, called a \emph{de Vries-Komornik number}, is a transcendental number (cf.~\cite{Kong_Li_2015}).

Instead of  investigating  the bifurcation set $\hub$ directly,  we consider two modified bifurcation sets:
\begin{align*}
\ub^L=\ub^L(M)&:=\set{q\in(1,M+1]:\dim_H(\us_q\setminus\us_p)>0 \textrm{ for any  }  p\in(1,q)};\\
\ub^R=\ub^R(M)&:=\set{q\in(1,M+1]: \dim_H(\us_r\setminus\us_q)>0 \textrm{ for any }  r\in(q, M+1]}.
\end{align*}
 The sets $\ub^L, \ub^R$ are called the \emph{left bifurcation set} and the \emph{right bifurcation set} of $F$, respectively. {In view of  \cite[Theorem 1.1]{Kong_Li_Lv_Vries2016},  the right bifurcation set $\ub^R$ is equal to the set of univoque bases such that $1$ has a unique expansion, i.e., $\ub^R=\ub$.} Clearly, $\hub\subset\ub^L$ and $\hub\subset\ub^R$. Furthermore,
 \[\ub^L\cap\ub^R=\hub\quad\textrm{and}\quad \ub^L\cup\ub^R=\overline{\hub}.\]
 {By   (\ref{e13}) it follows that  the difference set $\ub^L\setminus\hub$ consists of all left endpoints  of the plateaus in $(q_{KL}, M+1]$ of $F$, and hence it is countable.   Similarly, the difference set $\ub^R\setminus\hub$ consists of all right endpoints of the plateaus of $F$.} Therefore,
{\begin{equation}
 \label{e14}
 \begin{split}
 (1,M+1]\setminus\ub^L&=(1,q_{KL}]\cup\bigcup(q_0, q_0^*],\\
  (1,M+1]\setminus\ub^R&=(1,q_{KL})\cup\bigcup[q_0, q_0^*).
  \end{split}
 \end{equation}
 }

{Since the differences among $\hub, \ub^L$, $\ub^R=\ub$ and $\overline{\ub}$ are at most countable, the dimensional results obtained in this paper for $\ub=\ub^R$ also hold for $\hub, \ub^L$ and $\overline{\ub}$.}

Now we recall from \cite{Kong_Li_Lv_Vries2016} the following characterizations of the left and right bifurcation sets $\ub^L$ and $\ub^R$ respectively.

{\begin{theorem}[\cite{Kong_Li_Lv_Vries2016}]\label{th:11}\mbox{}

 \begin{enumerate}[{\rm(i)}]
 \item  $q\in\ub^L$ if and only if  
 $\dim_H(\ub\cap(p, q))>0$ {for any} $p\in(1,q).$
 
 \item $q\in\ub^R$ if and only if 
$\dim_H(\ub\cap(q, r))>0$ {for any} $ r\in(q, M+1].$
 \end{enumerate}
 
\end{theorem}

\begin{remark}\label{rem:12} 
Since $\hub=\ub^L\cap\ub^R$,  Theorem \ref{th:11}  also gives an equivalent condition for the bifurcation set $\hub$, i.e.,   $q\in\hub$ if and only if
\[
\dim_H(\ub\cap(p, q))>0\qquad\textrm{and}\qquad \dim_H(\ub\cap(q,r))>0
\]
for any $1<p<q<r\le M+1$.
\end{remark}
}

\subsection{Entropy bifurcation set $\bb$}
 For a symbolic subset $X\subset{\Omega}$ its \emph{topological entropy} is defined by
\[
h(X):=\liminf_{n\ra\f}\frac{\log\# B_n(X)}{n},
\]
where $B_n(X)$ denotes the set of all length $n$ subwords occurring in elements of $X$, and $\# A$ denotes the cardinality of a set $A$. {Here and throughout the paper we use base $M+1$ logarithms.} Recently, Komornik et al. showed in \cite{Komornik_Kong_Li_2015_1} (see also Lemma \ref{l25} below) that the function
\[
H: (1, M+1]\ra [0, {1}]; \qquad q\mapsto h(\us_q)
\]
is a Devil's staircase:
\begin{itemize}
\item $H$ is a continuous and {non-decreasing} function from $(1,M+1]$ onto $[0, {1}]$.
\item $H$ is {locally constant} Lebesgue almost everywhere  in $(1, M+1]$.
\end{itemize}

Let $\bb$ be the \emph{bifurcation set} of the entropy function $H$, defined by
\[
\bb=\bb(M):=\set{q\in(1,M+1]: H(p)\ne H(q)\textrm{ for any }p\ne q}.
\]
In \cite{AlcarazBarrera-Baker-Kong-2016} Alcaraz Barrera with  the second and third authors proved that $\bb\subset\ub$, and hence $\bb$ is of zero Lebesgue measure. They also showed   that $\bb$ has full Hausdorff dimension.  Furthermore, $\bb$ has no isolated points and can be written as
\begin{equation}\label{e15}
(1, M+1]\setminus\bb=(1,q_{KL}]\cup\bigcup[p_L, p_R],
\end{equation}
where the union on the right hand side is countable and pairwise disjoint. By the definition of the bifurcation set $\bb$ it follows that    each connected component $[p_L, p_R]$   is a maximal interval on which $H$ is constant. Thus each closed interval  $[p_L,p_R]$ is called a \emph{plateau} of $H$ (or an \emph{entropy plateau}). Furthermore, the left and right endpoints of each entropy plateau in $(q_{KL},  M+1]$ are both algebraic numbers (see also Lemma \ref{l31} below).

In analogy with $\ub^L$ and $\ub^R$ we also define two one-sided bifurcation sets of $H$:
\begin{align*}
\bb^L&=\bb^L(M):=\set{q\in(1, M+1]: H(p)<H(q) \textrm{ for any }p\in(1,q)};\\
\bb^R&=\bb^R(M):=\set{q\in(1, M+1]: H(r)>H(q)\textrm{ for any }r\in(q, M+1]}.
\end{align*}
We call $\bb^L$ and $\bb^R$  the \emph{left bifurcation set} and the \emph{right bifurcation set} of $H$, respectively. Comparing these sets with the bifurcation sets $\hub, \ub^L$ and $\ub^R$ of $F,$ we have analogous properties for the bifurcation sets $\bb, \bb^L$ and $\bb^R$.
For example,  $\bb\subset\bb^L$ and $\bb\subset\bb^R$. Furthermore,
\[\bb^L\cap\bb^R=\bb\qquad \textrm{and}\qquad \bb^L\cup\bb^R=\overline{\bb}.\]
{The difference set $\bb^L\setminus\bb$ consists of all left endpoints of the plateaus in $(q_{KL},  M+1]$ of $H$. Similarly, $\bb^R\setminus\bb$ consists of all right endpoints of the plateaus of $H$.} In other words, by (\ref{e15}) we have
\begin{align}
\label{e16}
\begin{split}
(1, M+1]\setminus\bb^L&=(1, q_{KL}]\cup\bigcup(p_L, p_R],\\
(1, M+1]\setminus\bb^R&=(1,q_{KL})\cup\bigcup[p_L, p_R).
\end{split}
\end{align}
 We emphasize that $M+1$ belongs to $\bb, \bb^L$ and $\bb^R$. Since $\bb\subset\hub$,  by (\ref{e14}) and (\ref{e16}) we also have
\[
\bb^L\subset\ub^L\qquad \textrm{and}\qquad\bb^R\subset\ub^R.
\]

Now we state our main results.
Inspired by the characterizations of $\ub^L$ and $\ub^R$ described in {Theorem  \ref{th:11}}, we characterize the left and right bifurcation sets $\bb^L$ and $\bb^R$ respectively.

\begin{main}
\label{main:1}
If $M=1$ or $M$ is even, the following statements are equivalent.
\begin{enumerate}[{\rm(i)}]
\item $q\in\bb^L$.
%\item $H(p)<H(q)$ for any $p\in(1,q)$.
\item $\dim_H(\us_q\setminus\us_p)=\dim_H\us_q>0$ for any $p\in(1,q)$.
\item $\lim_{p\nearrow q}\dim_H(\bb \cap(p, q))=\dim_H\u_q>0$.
\item $\lim_{p\nearrow q}\dim_H(\ub\cap(p,q))=\dim_H\u_q>0.$
\end{enumerate}
\end{main}

For odd $M\geq 3$ this theorem must be modified. This is due to the surprising presence of a single exceptional base $q_\star$ which is not an element of $\bb^L$, but for which (ii) and (iv) of Theorem \ref{main:1} nonetheless hold. Let
\begin{equation}
\label{e17}
{q_\star=} q_\star(M):=\begin{cases}
\frac{k+3+\sqrt{k^2+6k+1}}{2} & \qquad\mbox{if \quad$M=2k+1$},\\
\frac{k+3+\sqrt{k^2+6k-3}}{2} & \qquad\mbox{if \quad$M=2k$}.
\end{cases}
\end{equation}
(We will have use for $q_\star(M)$ with $M$ even later on.)

\begin{alternativetheorem}{\ref{main:1}}
\label{main-prime}
Suppose $M=2k+1\geq 3$. 
\begin{enumerate}[(a)]
\item $q\in\bb^L$ if and only if $\lim_{p\nearrow q}\dim_H(\bb \cap(p, q))=\dim_H\u_q>0$.
\item The following statements are equivalent:
\begin{enumerate}[{\rm(i)}]
\item $q\in\bb^L\cup\{q_\star(M)\}$.
\item $\dim_H(\us_q\setminus\us_p)=\dim_H\us_q>0$ for any $p\in(1,q)$.
\item $\lim_{p\nearrow q}\dim_H(\ub\cap(p,q))=\dim_H\u_q>0.$
\end{enumerate}
\end{enumerate}
\end{alternativetheorem}

The characterization of $\bb^R$ is more straightforward:

\begin{main}\label{main:2}
The following statements are equivalent for every $M\in\N$.
\begin{enumerate}[{\rm(i)}]
\item $q\in\bb^R$.
%\item $H(r)>H(q)$ for any $r\in(q,2]$.
\item $\dim_H(\us_r\setminus\us_q)=\dim_H\us_r>0$ for any $r\in(q,M+1]$.
\item $\lim_{r\searrow q}\dim_H(\bb\cap(q,r))=\dim_H\u_q>0$, or $q= q_{KL}$.
\item $\lim_{r\searrow q}\dim_H(\ub\cap(q, r))=\dim_H\u_q>0$, or $q= q_{KL}$.
\end{enumerate}
\end{main}

The asymmetry between the characterizations of $\bb^L$ and $\bb^R$ can be partially explained by the asymmetry of entropy plateaus. For instance, if $[p_L,p_R]$ is an entropy plateau, it follows from \cite[Lemma 4.10]{AlcarazBarrera-Baker-Kong-2016} that $p_L\in{\overline{\ub}\setminus \ub}$, whereas $p_R\in{\ub}$. Moreover, $p_R$ is a left {and right} accumulation point of ${\ub}$, but $p_L$ is not a right accumulation point of $\ub$. This helps explain why there is no counterpart in Theorem \ref{main:2} to the special base {$q_\star(M)$} of Theorem \ref{main-prime}.

\begin{remark}\label{rem:13} \mbox{}
\begin{enumerate}
\item  
Since $\bb=\bb^L\cap\bb^R$ and $q_{KL}\notin\bb$, Theorems \ref{main:1}, \ref{main-prime} and \ref{main:2} give equivalent conditions for the bifurcation set $\bb$. For example, {when $M=1$, $q\in\bb$}   if and only if
\[
  \lim_{p\nearrow q}\dim_H(\ub\cap(p,q))=\lim_{r\searrow q}\dim_H(\ub\cap(q,r))=\dim_H\u_q>0.
\]

\item 
In view of Lemma \ref{l311} {below}, we emphasize that the limits in statements (iii) and (iv) of  Theorems \ref{main:1} and \ref{main:2} are at most equal to $\dim_H\u_q$ for every $q\in(1, M+1]$. So, the theorems characterize when this largest possible value is attained. 

\end{enumerate}
\end{remark}

{ 
Since the sets $\ub$ and $\bb$ are of Lebesgue measure zero and nowhere dense, a natural measure of their distribution within the interval $(1, M+1]$ are the local dimension functions 
\[
\lim_{\de\ra 0}\dim_H(\ub\cap(q-\de, q+\de))\qquad\textrm{and}\qquad \lim_{\de\ra 0}\dim_H(\bb\cap(q-\de, q+\de)).
\]
In \cite[Theorem 2]{Kalle-Kong-Li-Lv-2016} it was shown that 
{\[
q\in\overline{\bb}\setminus\set{q_{KL}}\quad\Longleftrightarrow\quad \lim_{\de\ra 0}\dim_H(\bb\cap(q-\de, q+\de))=\dim_H\u_q>0.
\]}
As for the set $\ub$, we will show in Lemma \ref{l311} below that 
\begin{equation}\label{e18}
\lim_{\de\ra 0}\dim_H(\ub\cap(q-\de, q+\de))\le \dim_H\u_q\qquad\textrm{for all}\quad q\in(1, M+1].
\end{equation}
Observe that {$q_\star(M)\in\bb^R$} for $M=2k+1\ge 3$. {(See Lemma \ref{l31} below.) Thus} Theorems \ref{main:1}, \ref{main-prime} and \ref{main:2} imply that the upper bound $\dim_H\u_q$ for the limit in (\ref{e18}) is attained if and only if $q\in\overline{\bb}$. Precisely:

\begin{cmain}\label{cor:3}
$q\in\overline{\bb}\setminus\set{q_{KL}}$ if and only if 
\[\lim_{\de\ra 0}\dim_H(\ub\cap(q-\de, q+\de))=\dim_H\u_q>0.\]
\end{cmain}
}

{Clearly, $\lim_{\de\ra 0}\dim_H(\ub\cap(q-\de, q+\de))=0$ when $q\not\in\overline{\ub}$. It is interesting to ask which values this limit can take for $q\in\overline{\ub}\backslash\overline{\bb}$. This may be the subject of a future paper.
%In terms of Corollary \ref{cor:3} it would be interesting to calculate the local dimension $\lim_{\de\ra 0}\dim_H(\ub\cap(q-\de, q+\de))$ for any $q\in\overline{\ub}\setminus\overline{\bb}$. 
}

%By Theorems \ref{th:11}, \ref{th:12}, \ref{main:1} and \ref{main:2}  we also have the following characterizations of the difference sets $\ub^L\setminus\bb^L$ and $\ub^R\setminus\bb^R$, respectively.
%
%{\begin{cmain}
%\label{cor:3}\mbox{}
%
%\begin{enumerate}[{\rm(i)}]
%\item $q\in\ub^L\setminus\bb^L$ or $q\in\ub^L\setminus(\bb^L\cup\set{q_\star})$ for $M=2k+1\ge 3$ if and only if 
% \[0<\dim_H(\us_q\setminus\us_p)<\dim_H\us_q\qquad \textrm{for all sufficiently large}\quad p\in(1,q).\]
%
%\item $q\in\ub^R\setminus\bb^R$
%if and only if 
%\[0<\dim_H(\us_r\setminus\us_q)<\dim_H\us_r\qquad \textrm{for all sufficiently small}\quad r\in(q, M+1].\]
%
% \end{enumerate}
%\end{cmain}
%}

\subsection{The difference set $\ub \setminus \bb$}

Note that $\bb\subset {\ub}$, and both are Lebesgue null sets of full Hausdorff dimension.  Furthermore, $\ub\setminus\bb$ is a dense subset  of $\ub$. So the box dimension of $\ub\setminus\bb$ is given by
\[\dim_B(\ub\setminus\bb)=\dim_B(\overline{\ub\setminus\bb})=\dim_B {\overline{\ub}}=1. \]
    On the other hand, our next result shows that the Hausdorff dimension of $\ub\setminus\bb$ is significantly smaller than one.
		
{\begin{main}
\label{main:4}\mbox{}

\begin{enumerate}[{\rm(i)}]
\item If $M=1$, then
\[\dim_H(\ub\setminus\bb)=\frac{\log 2}{3\log\lambda^*}\approx 0.368699,\]
where $\lambda^*\approx 1.87135$ is the unique root in $(1,2)$ of the equation $x^5-x^4-x^3-2x^2+x+1=0$.

\item If $M=2$, then 
\[
\dim_H(\ub\setminus\bb)=\frac{\log 2}{2\log\ga^*}\approx 0.339607,
\]
where $\ga^*\approx 2.77462$ is the unique root in $(2,3)$ of the equation $x^4-2x^3-3x^2+2x+1=0$.

\item If $M\ge 3$, then 
\[
\dim_H(\ub\setminus\bb)=\frac{\log 2}{\log q_\star(M)},
%\begin{cases}
%\frac{\log 2}{\log ({k+3+\sqrt{k^2+6k+1}})-\log {2}} & \textrm{if $M=2k+1$},\\
%\frac{\log 2}{\log ({k+3+\sqrt{k^2+6k-3}})-\log {2}} & \textrm{if $M=2k$}.
%\end{cases}
\]
where $q_\star(M)$ is given by \eqref{e17}. 
\end{enumerate}
\end{main}
}

Table \ref{tab:1} below lists the values of $\dim_H(\ub\setminus\bb)$ for $1\leq M\leq 8$. For large $M$ we have {by Theorem \ref{main:4} (iii)} the simple approximation $\dim_H(\ub\setminus\bb)\approx\log 2/\log(k+3)$, where $k$ is the greatest integer less than or equal to $M/2$. This systematically underestimates the true value, with an error slowly tending to zero. {Observe also that $\dim_H(\ub\setminus\bb) \to 0$ as $M\to\f$.}

\begin{table}[h!]
%\begin{center}
\begin{tabular}{c|c|c|c|c|c|c|c|c}
\hline
$M$&1&2&3&4&5&6&7&8\\
\hline
$\dim_H(\ub\setminus\bb)$&0.3687& 0.3396& 0.5645& 0.4750& 0.4567& 0.4088&0.4005&0.3091\\
\hline
\end{tabular}
\bigskip
\caption{The numerical calculation of $\dim_H(\ub\setminus\bb)$ for $M=1,\ldots, 8$.}
%\end{center}
\label{tab:1}
\end{table}

In \cite{Kalle-Kong-Li-Lv-2016}, Kalle et al.~showed that $\dim_H({\ub}\cap(1,t])=\max_{q\leq t}\dim_H \pazocal{U}_q$ for all $t>1$, and they asked whether more generally it is possible to calculate $\dim_H({\ub}\cap[t_1,t_2])$ for any interval $[t_1,t_2]$. In the process of proving Theorem \ref{main:4}, we give a partial answer to their question by computing the Hausdorff dimension of the intersection of ${\ub}$ with any entropy plateau $[p_L,p_R]$ {(see Theorem \ref{thm:41})}. 

The rest of the paper is arranged as follows. In Section \ref{s2} we recall some results from unique $q$-expansions, and give the Hausdorff dimension of the symbolic univoque set $\us_q$ (see Lemma \ref{l28}). Based on these observations we characterize the left and right bifurcation sets $\bb^L$ and $\bb^R$ in Section \ref{s3}, by proving Theorems \ref{main:1}, \ref{main-prime} and \ref{main:2}. In Section \ref{s4} we prove Theorem \ref{main:4}.

\section{Unique expansions}\label{s2}

In this section we will describe the symbolic univoque set $\us_q$ and calculate its Hausdorff dimension. {Recall that $\Omega=\{0,1,\dots,M\}^\N$.
Let} $\si$ be the \emph{left shift} on $\Omega$ defined by $\si((c_i))=(c_{i+1})$. Then $(\Omega,\si)$ is a \emph{full shift}. By a \emph{word} $\c$ we mean a finite string of digits $\c=c_1\ldots c_n$ with each digit $c_i\in {\{0,1,\dots,M\}}$. For two words $\c=c_1\ldots c_m$ and $\d=d_1\ldots d_n$ we denote by $\c\d=c_1\ldots c_m d_1\ldots d_n$ their concatenation. For a positive integer $n$ we write $\c^n=\c\cdots\c$ for the $n$-fold concatenation of $\c$ with itself. Furthermore, we write $\c^\f=\c\c\cdots$ for the infinite periodic sequence with period block $\c$. For a word $\c=c_1\ldots c_m$ we set $\c^+:=c_1\ldots c_{m-1}(c_m+1)$ if $c_m<M$, and set $\c^-:=c_1\ldots c_{m-1}(c_m-1)$ if $c_m>0$.
 Furthermore, we define the \emph{reflection} of the word $\c$ by $\overline{\c}:=(M-c_1)(M-c_2)\cdots(M-c_m)$. Clearly, $\c^+, \c^-$ and $\overline{\c}$ are all words with digits from ${\{0,1,\dots,M\}}$. For a sequence $(c_i)\in\Omega$ its reflection is also a sequence in $\Omega$ defined by $\overline{(c_i)}=(M-c_1)(M-c_2)\cdots$.

 Throughout the paper we will use the \emph{lexicographical ordering} $\prec, \lle, \succ$ and $\lge$ between sequences and words. More precisely, for two sequences $(c_i), (d_i)\in\Omega$ we say $(c_i)\prec (d_i)$ or $(d_i)\succ (c_i)$ if there exists an integer $n\ge 1$ such that $c_1\ldots c_{n-1}=d_1\ldots d_{n-1}$ and $c_n<d_n$. Furthermore, we say $(c_i)\lle (d_i)$ if $(c_i)\prec (d_i)$ or $(c_i)=(d_i)$. Similarly, for two words $\c$ and $\d$ we say $\c\prec \d$ or $\d\succ\c$ if $\c0^\f\prec \d0^\f$.

Let $q\in(1,M+1]$.  Recall that $\us_q$ is the symbolic univoque set which contains all sequences $(x_i)\in\Omega$ such that $(x_i)$ is the unique $q$-expansion of   $\pi_q((x_i))$. Here
 $\pi_q$ is the projection map defined in (\ref{e11}).  The description of $\us_q$ is based on the \emph{quasi-greedy} $q$-expansion  of $1$, denoted by $\al(q)=\al_1(q)\al_2(q)\ldots$,  which is the lexicographically largest $q$-expansion of $1$ not ending with $0^\f$ (cf.~\cite{Daroczy_Katai_1993}). The following characterization of $\al(q)$ was given in \cite[Theorem 2.2]{Baiocchi_Komornik_2007} (see also \cite[Proposition 2.3]{Vries-Komornik-Loreti-2016}).

 \begin{lemma}\label{l21}
 The map $q\mapsto \al(q)$ is a strictly increasing bijection from $(1, M+1]$ onto the set of all sequences $(a_i)\in\Omega$ not ending with $0^\f$ and satisfying
 \[
 a_{n+1}a_{n+2}\ldots \lle a_1a_2\ldots \qquad\textrm{for all }\quad n\ge 0.
 \]
 Furthermore, the map $q\mapsto \al(q)$
 is left-continuous.
  \end{lemma}
	
{\begin{remark}\label{rem:22}
  Let ${\mathbf A}:=\set{\al(q): q\in(1,M+1]}$. Then Lemma \ref{l21} implies that the inverse  map 
  \[\al^{-1}: {\mathbf A}\ra (1,M+1];\qquad (a_i)\mapsto \al^{-1}((a_i)) \]
  is bijective and strictly increasing. Furthermore, we can even show that $\al^{-1}$ is continuous; see the proof of Lemma \ref{l36} below. 
  \end{remark}
}

Based on the quasi-greedy expansion $\al(q)$  we give the lexicographic characterization of the symbolic univoque set $\us_q$,  which was  essentially established by Parry \cite{Parry_1960} (see also \cite{Komornik_Kong_Li_2015_1}).

\begin{lemma}
\label{l23}
Let $q\in(1,M+1]$. Then $(x_i)\in\us_q$ if and only if
\[\left\{\begin{array}{lll}
x_{n+1}x_{n+2}\ldots \prec \al(q)&\quad\textrm{whenever}& x_n<M,\\
x_{n+1}x_{n+2}\ldots \succ \overline{\al(q)}&\quad\textrm{whenever}& x_n>0.
\end{array}\right.
\]
\end{lemma}

Note by Lemma \ref{l21} that when $q$ is increasing the quasi-greedy expansion $\al(q)$ is also increasing in the lexicographical ordering.  By Lemma \ref{l23} it follows that the set-valued map $q\mapsto \us_q$ is also increasing, i.e., $\us_p\subseteq\us_q$ when $p<q$.

Recall from \cite{Komornik_Loreti_2002} that the Komornik-Loreti constant  $q_{KL}=q_{KL}(M)$ is the smallest element of $\ub^R$, and satisfies
\begin{equation}\label{eq:21}
\al(q_{KL})=\la_1\la_2\ldots,
\end{equation}
where for each $i\ge 1$,
\begin{equation} 
\label{eq:22}
\la_i=\la_i(M):=\begin{cases}
k+\tau_i-\tau_{i-1} & \qquad\textrm{if \quad$M=2k$},\\
k+\tau_i & \qquad\textrm{if \quad$M=2k+1$}.
\end{cases}
\end{equation}
Here $(\tau_i)_{i=0}^\f=0110100110010110\ldots$ is the classical \emph{Thue-Morse sequence} (cf.~\cite{Allouche_Shallit_1999}). We emphasize that the sequence $(\la_i)$ depends on $M$. {The following recursive relation of $(\la_i)$ was established in \cite{Komornik_Loreti_2002} (see also \cite{Kong_Li_2015}): 
\begin{equation}\label{eq:23}
\la_{2^n+1}\ldots \la_{2^{n+1}}=\overline{\la_1\ldots \la_{2^n}}\,^+\qquad\textrm{for all}\quad n\ge 0.
\end{equation}
By  (\ref{eq:21}) and (\ref{eq:22}) it follows that $q_{KL}(M)\ge (M+2)/2$ for all $M\ge 1$ (see also \cite{Baker-2014}), and the map $M\mapsto q_{KL}(M)$ is strictly increasing. 
 
\begin{example}
\label{ex:24} 
{The following values of $q_{KL}(M)$ will be needed in the proof of Theorem \ref{main:4} in Section \ref{s4}.}
\begin{enumerate}
\item {Let }$M=1$. Then by (\ref{eq:22}) we have $\la_1=1$. By (\ref{eq:21}) and (\ref{eq:23}) it follows that 
\[
\al(q_{KL}(1))=1101\,0011\;00101101\ldots=(\tau_i)_{i=1}^\f.
\]
This gives  $q_{KL}(1)\approx 1.78723$.

\item {Let }$M=2$. Then by (\ref{eq:22}) we have $\la_1=2$, and  by (\ref{eq:21}) and (\ref{eq:23}) that 
\[
\al(q_{KL}(2))=2102\,0121\;01202102\ldots.
\]
So $q_{KL}(2)\approx 2.53595$.

\item {Let }$M=3$. Then by (\ref{eq:22}) we have $\la_1=2$, and  by (\ref{eq:21}) and (\ref{eq:23}) that
\[
\al(q_{KL}(3))=2212\,1122\;11212212\ldots.
\]
Hence, $q_{KL}(3)\approx 2.91002$. 
\end{enumerate}
\end{example}
}

Now we recall from \cite{Komornik_Kong_Li_2015_1} the following result for the Hausdorff dimension of the univoque set $\u_q$.

\begin{lemma} \label{l25}\mbox{}
\begin{enumerate}
\item[{\rm(i)}]  For any $q\in(1,M+1]$ we have
\[
\dim_H\u_q=\frac{h(\us_q)}{\log q}.
\]
\item[{\rm(ii)}]  The entropy function $H: q\mapsto h(\us_q)$ is a Devil's staircase in $(1,M+1]$:
\begin{itemize}
\item $H$ is {non-decreasing} and continuous from $(1,M+1]$ onto $[0, {1}]$;
\item $H$ is { locally constant } almost everywhere in $(1,M+1]$.
\end{itemize}
\item [{\rm(iii)}] $H(q)>0$ if and only if $q>q_{KL}$. Furthermore, $H(q)=\log (M+1)$ if and only if $q=M+1$.
\end{enumerate}
\end{lemma}

We also need the following lemma for the Hausdorff dimension under H\"{o}lder continuous maps (cf.~\cite{Falconer_1990}).

\begin{lemma}\label{l26}
Let $f: (X, \rho_X)\ra (Y, \rho_Y)$ be a H\"{o}lder map between two metric spaces, i.e., there exist two constants $C>0$ and $\xi>0$ such that
\[
\rho_Y(f(x), f(y))\le C \rho_X(x, y)^\xi\quad\textrm{for any }x, y\in X.
\]
Then $\dim_H f(X)\le \frac{1}{\xi}\dim_H X.$
\end{lemma}

Recall the metric $\rho$ from (\ref{e12}). It will be convenient to introduce a more general family of (mutually equivalent) metrics $\{\rho_q:q>1\}$ {on $\Omega$} defined by
\begin{equation*}
\rho_q((c_i),(d_i)):=q^{-\inf\{i\ge 1:c_i\neq d_i\}}, \qquad q>1.
\end{equation*}
Then $(\Omega, \rho_q)$ is a compact metric space.
Let $\dim_H^{(q)}$ denote Hausdorff dimension on $\Omega$ with respect to the metric $\rho_q$, so 
\[\dim_H^{(M+1)}E=\dim_H E\]
 for any subset $E\subseteq\Omega$. 
For $p>1$ and $q>1$,
\begin{equation*}
\rho_q((c_i),(d_i))=\rho_p((c_i),(d_i))^{\log q/\log p},
\end{equation*}
and {by Lemma \ref{l26}} this gives the useful relationship 
\begin{equation}
\dim_H^{(p)}E=\frac{\log q}{\log p}\dim_H^{(q)}E, \qquad E\subseteq\Omega.
\label{eq:24}
\end{equation}

The following result is well known (see \cite[Lemma 2.7]{Jordan-Shmerkin-Solomyak-2011} or \cite[Lemma 2.2]{Allaart-2017}):

\begin{lemma} \label{l27}
For each $q\in(1, M+1)$, the map $\pi_q$ is Lipschitz on $(\Omega,\rho_q)$, and the restriction
\begin{equation*}
\pi_q: (\mathbf{U}_q,\rho_q)\to (\pazocal{U}_q,|.|);\qquad \pi_q((x_i))=\sum_{i=1}^\infty \frac{x_i}{q^i}
\end{equation*}
is bi-Lipschitz, where $|.|$ denotes the Euclidean metric on $\R$. 
\end{lemma}

{Observe that the Hausdorff dimension does not exceed the lower box dimension (cf.~\cite{Falconer_1990}). This implies that $
\dim_H E\le h(E)
$  for any set $E\subset\Omega$.
{Using} Lemmas \ref{l25}--\ref{l27} we show that equality holds for $\us_q$.}

\begin{lemma} \label{l28}
Let $q\in(1, M+1]$. Then
\[
\dim_H\us_q={h(\us_q)}.
\]
\end{lemma}

{\begin{proof}
For $q=M+1$, one checks easily that 
\[
\dim_H\us_{M+1}={h(\us_{M+1})}=1.
\]
Let $q\in(1, M+1)$. By Lemmas \ref{l27} and \ref{l26}, $\dim_H^{(q)}\us_q=\dim_H \pazocal{U}_q$. So \eqref{eq:24},  {Lemmas \ref{l27}  and \ref{l25}} give
{\begin{align*}
\dim_H\us_q=\dim_H^{(M+1)}\us_q=\frac{\log q}{\log (M+1)}\dim_H^{(q)}\us_q&= {\log q} \dim_H \pazocal{U}_q= {h(\us_q)},
\end{align*}
}
as desired. We emphasize that the base for our logarithms is $M+1$. 
\end{proof}
}

Note that the symbolic univoque set $\us_q$ is not always closed. Inspired by the works of de Vries and Komornik \cite{DeVries_Komornik_2008} and Komornik et al.~\cite{Komornik_Kong_Li_2015_1} we introduce the set
\begin{equation}\label{eq:25}
\vs_q:=\set{(x_i)\in\Omega: \overline{\al(q)}\lle x_{n+1}x_{n+2}\ldots \lle \al(q)\textrm{ for all }n\ge 0}.
\end{equation}

We have the following relationship between $\vs_q$ and $\us_q$.

\begin{lemma}
\label{l29}
For any $0<p<q\le M+1$ we have
\[
\dim_H\vs_q=\dim_H\us_q \qquad\mbox{and} \qquad \dim_H(\vs_q\setminus\vs_p)=\dim_H(\us_q\setminus\us_p).
\]
\end{lemma}

\begin{proof}
By Lemma \ref{l23} it follows that for each $q\in(1, 2]$ the set $\us_q$ is a countable union of affine copies of $\vs_q$ up to a countable set (see also \cite[Lemma 3.2]{Kalle-Kong-Li-Lv-2016}), i.e., there exists a sequence of affine maps $\set{g_i}_{i=1}^\f$ on $\Omega$ of the form
\[
x_1x_2\ldots\mapsto a x_1x_2\ldots,\qquad x_1x_2\ldots \mapsto M^m b x_1x_2\ldots \qquad\mbox{or}\qquad x_1x_2\ldots \mapsto 0^m c x_1x_2\ldots,
\]
where  $a\in\set{1,2,\ldots, M-1}$, $b\in\set{0,1,\ldots, M-1}$, $c\in\set{1,2,\ldots, M}$  and $m=1,2,\ldots$, such that
\begin{equation}
\label{eq:26}
\us_q\sim\bigcup_{i=1}^\f g_i(\vs_q),
\end{equation}
where we write $A\sim B$ to mean that the symmetric difference $A \bigtriangleup B$ is at most countable.
 Since the Hausdorff dimension is stable under affine maps (cf.~\cite{Falconer_1990}), this implies  $\dim_H\vs_q=\dim_H\us_q$.

Furthermore, for any $1<p<q\le M+1$ we have $\us_p\subseteq\us_q$ and $\vs_p\subseteq\vs_q$, so $g_i(\vs_p)\subseteq g_i(\vs_q)$ for all $i\ge 1$. 
Note that for $i\ne j$ the intersection  $g_i(\vs_q)\cap g_j(\vs_q)=\emptyset$.
Then by (\ref{eq:26}) it follows that
\begin{align*}
\us_q\setminus\us_p &\sim \bigcup_{i=1}^\f g_i(\vs_q)\setminus\bigcup_{i=1}^\f g_i(\vs_p)\\
&=\bigcup_{i=1}^\f \big(g_i(\vs_q)\setminus g_i(\vs_p)\big)=\bigcup_{i=1}^\f g_i(\vs_q\setminus\vs_p).
\end{align*}
We conclude that $\dim_H(\us_q\setminus\us_p)=\dim_H(\vs_q\setminus\vs_p)$.
\end{proof}

\section{Characterizations of  $\bb^L$ and $\bb^R$}\label{s3}

Recall from \eqref{e16} that $\bb^L$ and $\bb^R$ are the left and right bifurcation sets of $H$. In this section we will characterize the sets $\bb^L$ and $\bb^R$, and prove Theorems \ref{main:1}, \ref{main-prime} and \ref{main:2}. Since the theorems are very similar, we will prove only Theorem \ref{main:1} in full detail, and comment briefly on the proofs of Theorems \ref{main-prime} and \ref{main:2}.
 
Recall the definition of $q_\star(M)$ from \eqref{e17}. Its significance derives from the fact that 
{\[
\al(q_\star(M))=\left\{
\begin{array}{lll}
(k+2)k^\f&\qquad\textrm{if}\quad& M=2k+1,\\
(k+2)(k-1)^\f&\qquad\textrm{if}\quad & M=2k.
\end{array}
\right.
\]
}
By (\ref{eq:21}) and Lemma \ref{l21} it follows in particular that $q_\star(M)>q_{KL}$.

 Recall  that a closed interval $[p_L, p_R]\subseteq(q_{KL}, M+1]$ is an entropy plateau if it is a maximal interval on which $H$ is constant. 
 The following lemma was implicitly proven in \cite{AlcarazBarrera-Baker-Kong-2016}.
 
\begin{lemma}\label{l31}
Let $[p_L, p_R]\subset(q_{KL}, M+1]$ be an entropy plateau.
\begin{enumerate}[{\rm(i)}]
\item Then there exists a word $a_1\ldots a_m$ satisfying  $\overline{a_1}<a_1$ and
\[
\overline{a_1\ldots a_{m-i}}\lle a_{i+1}\ldots a_m\prec a_1\ldots a_{m-i}\qquad\textrm{for all}\quad 1\le i<m,
\] such that
\begin{equation*}
\label{e32}
\al(p_L)=(a_1\ldots a_m)^\f\qquad\textrm{and}\qquad \al(p_R)=a_1\ldots a_m^+(\overline{a_1\ldots a_m})^\f.
\end{equation*}
\item Let $m\ge 1$ be defined as in (i). Then 
\[
h(\us_{p_L})\ge \frac{\log 2}{m},
\]
where equality holds if and only if $M=2k+1\ge 3$ and $[p_L, p_R]=[k+2, q_\star(M)]$.
%\[[p_L, p_R]=[k+2, q_\star]\qquad \textrm{for}\quad M=2k+1\ge 3.\]
\end{enumerate}
\end{lemma}

\begin{proof}
Part (i) was established in \cite[Theorem 2 and Lemma 4.1]{AlcarazBarrera-Baker-Kong-2016}. Part (ii) was implicitly given in the proofs of \cite[Lemmas 5.1 and 5.5]{AlcarazBarrera-Baker-Kong-2016}. It is shown there that $h(\us_{p_L})>\log 2/m$ when $m\geq 2$. If $m=1$, then $\al(p_L)=a_1^\f$ for some $a_1\geq (M+1)/2$, and
\[
h(\us_{p_L})=\log(2a_1-M+1).
\]
(See \cite[Example 5.13]{AlcarazBarrera-Baker-Kong-2016}.) It follows that $h(\us_{p_L})=\log 2/m$ if and only if $m=1$, $M=2k+1\geq 3$ and $a_1=k+1$, in which case
\[
\al(p_L)=(k+1)^\f\qquad\textrm{and}\qquad \al(p_R)=(k+2)k^\f,
\]
or equivalently, 
\[[p_L, p_R]=\left[k+2, \frac{k+3+\sqrt{k^2+6k+1}}{2}\right]=[k+2, q_\star(M)]\]
for $M=2k+1\geq 3$.
\end{proof}

{
\begin{remark}
\label{rem:32}
We point out that  the condition in Lemma \ref{l31} (i) is not a sufficient condition for $[p_L, p_R]\subset(q_{KL}, M+1]$ being an entropy plateau. For a complete characterization of entropy plateaus we refer to  \cite[Theorem 2]{AlcarazBarrera-Baker-Kong-2016}. However, if $[p_L,p_R]$ is an interval satisfying the conditions of Lemma \ref{l31}, then $[p_L,p_R]$ is either an entropy plateau or else it is contained in some entropy plateau {(see Example \ref{ex1} below)}. 
We refer to \cite{AlcarazBarrera-Baker-Kong-2016} for more details.
\end{remark}
}

{\begin{example}\label{ex1}
Take $M=1$ and let $a_1\ldots a_m=1^{m-1}0$ with $m\ge 3$. Then the word $a_1\ldots a_m$ satisfies the inequalities in Lemma \ref{l31} (i), and the interval $[p_L, p_R]$ is indeed an entropy plateau, where $\al(p_L)=(1^{m-1}0)^\f$ and $\al(p_R)=1^m(0^{m-1}1)^\f$. 

On the other hand, take the word $b_1\ldots b_{2m}=1^m0^m$. One can also check that $b_1\ldots b_{2m}$ satisfies the inequalities in Lemma \ref{l31} (i). However, the corresponding interval $[q_L, q_R]$ is a proper subset of $[p_L, p_R]$ and hence not an entropy plateau, where $\al(q_L)=(b_1\ldots b_{2m})^\f$ and $\al(q_R)=b_1\ldots b_{2m}^+(\overline{b_1\ldots b_{2m}})^\f$.
\end{example}
}

\begin{definition}\label{def:33}
If $[p_L,p_R]$ is an entropy plateau with $\alpha(p_L)=(a_1\dots a_m)^\infty$ {and $\al(p_R)=a_1\ldots a_m^+(\overline{a_1\ldots a_m})^\f$}, we shall call $[p_L,p_R]$ an {\em entropy plateau of period $m$}.
\end{definition} 

Recall that $\ub$ is the {set of univoque bases $q\in(1,M+1]$ such that $1$ has a unique $q$-expansion}. The following characterization of its topological closure $\overline{\ub}$ was established in  \cite{Komornik_Loreti_2007} (see also \cite{Vries-Komornik-Loreti-2016}).

\begin{lemma}
\label{l34}
 $q\in\overline{\ub}$ if and only if 
 \[\overline{\al(q)}\prec\si^n(\al(q))\lle \al(q)\qquad \textrm{for all}\quad  n\ge 1.\]
\end{lemma}

Lemma \ref{l21} states that the map {$\alpha: q\mapsto \al(q)$} is left-continuous on $(1,M+1]$. The following lemma strengthens this result when $\alpha$ is restricted to $\overline{\mathscr{U}}$.

\begin{lemma}
\label{l35}
{Let  $I=[p, q]\subset(1,M+1)$. Then the map $\al$ is Lipschitz on $\overline{\mathscr{U}}\cap I$ with respect to the metric $\rho_q$.}
\end{lemma}

\begin{proof}
Fix $1<p<q<M+1$.
We will show something slightly stronger, namely that there is a constant $C=C(p,q)$ such that for any $p\le p_1<p_2\leq q$ with $p_2\in \overline{\mathscr{U}}$, 
\[
\rho_q(\alpha(p_1),\alpha(p_2)) \leq C|p_2-p_1|.
\]

Let $p\le p_1<p_2\le q$ and $p_2\in \overline{\mathscr{U}}$. Then by Lemma \ref{l21} we have $\al(p_1)\prec \al(p_2)$. So there exists $n\ge 1$ such that $\al_1(p_1)\ldots \al_{n-1}(p_1)=\al_1(p_2)\ldots \al_{n-1}(p_2)$ and $\al_n(p_1)<\al_n(p_2)$. Since $q<M+1$,  we have $\al(q)\prec M^\infty$. Hence there exists a large integer $N\ge 1$, depending only on $q$, such that $\al(p_2)\lle \al(q)\lle M^{N-1}0^\f$. Since $p_2\in\overline{\ub}$, it follows by Lemma \ref{l34} that
 \[
 \al_{n+1}(p_2)\al_{n+2}(p_2)\ldots{\succ} \overline{\al(p_2)}\lge 0^{N-1}M^\f.
 \]
 This implies
 \[
 1=\sum_{i=1}^\f\frac{\al_i(p_2)}{p_2^i}>\sum_{i=1}^n\frac{\al_i(p_2)}{p_2^i}+\frac{1}{p_2^{n+N}}.
 \]
Therefore,
 \begin{align*} 
 \frac{1}{p_2^{n+N}}\le 1-\sum_{i=1}^n\frac{\al_i(p_2)}{p_2^i}&=\sum_{i=1}^\f\frac{\al_i(p_1)}{p_1^i}-\sum_{i=1}^n\frac{\al_i(p_2)}{p_2^i}\\
 &\le \sum_{i=1}^n\left(\frac{\al_i(p_2)}{p_1^i}-\frac{\al_i(p_2)}{p_2^i}\right)\\
 &\le \sum_{i=1}^\f\left(\frac{M}{p_1^i}-\frac{M}{p_2^i}\right)=\frac{M|p_2-p_1|}{(p_1-1)(p_2-1)}\\
 &\leq\frac{M|p_2-p_1|}{(p-1)^2}.
 \end{align*}
 Here the second inequality follows by using $\al_1(p_1)\ldots \al_{n-1}(p_1)=\al_1(p_2)\ldots \al_{n-1}(p_2)$, $\al_n(p_1)<\al_n(p_2)$ and the property of quasi-greedy expansion that $\sum_{i=1}^\f\al_{n+i}(p_1)/ p_1^i\le 1$. Therefore, we obtain
\begin{equation*}  
\rho_q(\alpha(p_1),\alpha(p_2))=q^{-n}\leq p_2^{-n}\leq \frac{ Mq^N}{(p-1)^2}|p_2-p_1|.
\end{equation*}
The proof is complete.
\end{proof}
 
The following dimension estimates will be very useful throughout the paper: 

\begin{lemma}
\label{l36}
For any interval $I=[p,q]\subseteq (1,M+1)$, 
\[
\dim_H\pi_{q}(\us_I) \le \dim_H(\overline{\mathscr{U}}\cap I) \leq \frac{h(\mathbf{U}_I)}{\log p},
\]
where $\mathbf{U}_I:=\{\alpha(\ell): \ell\in\overline{\mathscr{U}}\cap I\}$.
\end{lemma}

\begin{proof}
Fix an interval $I=[p,q]\subseteq (1,M+1)$. We may view the map $\pi_q\circ\al: \overline{\ub}\cap I\to \R$ as the composition of the maps $\al:\overline{\ub}\cap I\to (\us_I,\varphi_q)$ and $\pi_q: (\us_I,\varphi_q)\to \R$. The first map is Lipschitz by Lemma \ref{l35}, and the second is Lipschitz by Lemma \ref{l27}, since $\us_I\subset\us_q$. Therefore, the composition $\pi_q\circ \al$ is Lipschitz. {Using Lemma \ref{l26}, this implies the first inequality.}

The second inequality is proved as follows. Let $p\leq p_1<p_2\leq q$. Then $\alpha(p_1)\prec \alpha(p_2)$ by Lemma \ref{l21}, so there is a number $n\in\N$ such that $\al_1(p_1)\ldots \al_{n-1}(p_1)=\al_1(p_2)\ldots \al_{n-1}(p_2)$ and $\al_n(p_1)<\al_n(p_2)$. As in the proof of Lemma 4.3 in \cite{Kalle-Kong-Li-Lv-2016}, we then have
\begin{align*}
p_2-p_1 &= \sum_{i=1}^\infty \frac{\al_i(p_2)}{p_2^{i-1}}-\sum_{i=1}^\infty \frac{\al_i(p_1)}{p_1^{i-1}}\\
&\leq \sum_{i=1}^{n-1}\left(\frac{\al_i(p_2)}{p_2^{i-1}}-\frac{\al_i(p_1)}{p_1^{i-1}}\right)+\sum_{i=n}^\infty \frac{\al_i(p_2)}{p_2^{i-1}}\\
&\leq p_2^{2-n} \leq (M+1)^2 p^{-n},
\end{align*}
where the second inequality follows by the property of the quasi-greedy expansion $\alpha(p_2)$ of $1$. We conclude that
\[\rho(\al(p_1),\al(p_2))=(M+1)^{-n}=p^{-n/\log p} \geq \left(\frac{p_2-p_1}{(M+1)^2}\right)^{1/\log p},\]
in other words, the map $\alpha^{-1}$ is H\"older continuous with exponent $\log p$ on the set $\{\al(\ell):p\leq \ell\leq q\}$. It follows using Lemma \ref{l26} that
\begin{align*}
\dim_H(\overline{\mathscr{U}}\cap I) &= \dim_H(\al^{-1}(\mathbf{U}_I)) \leq \frac{ \dim_H \mathbf{U}_I}{\log p}{\le\frac{h(\mathbf{U}_I)}{\log p}},
\end{align*}
completing the proof.
\end{proof}

{Let $[p_L, p_R]\subset(q_{KL}, M+1]$ be an entropy plateau such that $\al(p_L)=(a_1\ldots a_m)^\f$ and $\al(p_R)=a_1\ldots a_m^+(\overline{a_1\ldots a_m})^\f$.} The proofs of the following two propositions use the sofic subshift $(X_{\mathcal G}, \si)$ represented by the labeled graph
$\mathcal G=(G, \mathcal L)$ in Figure \ref{fig:1} (cf.~\cite[Chapter 3]{Lind_Marcus_1995}).

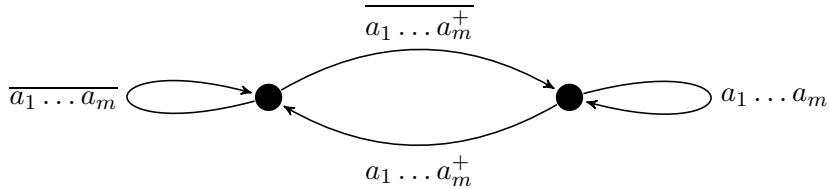
\begin{figure}[h!]
  \centering
  % Requires \usepackage{graphicx}
 \begin{tikzpicture}[->,>=stealth',shorten >=1pt,auto,node distance=4cm,
                    semithick]

  \tikzstyle{every state}=[minimum size=0pt,fill=black,draw=none,text=black]

  \node[state] (A)                    { };
  \node[state]         (B) [ right of=A] { };

  \path[->,every loop/.style={min distance=0mm, looseness=60}]
   (A) edge [loop left,->]  node {$\overline{a_1\ldots a_m}$} (A)
            edge  [bend left]   node {$\overline{a_1\ldots a_m^+}$} (B)

        (B) edge [loop right] node {$a_1\ldots a_m$} (B)
            edge  [bend left]            node {$a_1\ldots a_m^+$} (A);
\end{tikzpicture}
  \caption{The picture of the labeled graph $\mathcal G=(G, \mathcal L)$.}\label{fig:1}
\end{figure}
 
{We emphasize that $(X_{\mathcal G}, \si)$ is {in fact} a subshift of finite type over the states  
\[a_1\ldots a_m,\qquad a_1\ldots a_m^+,\qquad \overline{a_1\ldots a_m}\qquad\textrm{and}\qquad \overline{a_1\ldots a_m^+}\]
with adjacency matrix
\[
A_{\mathcal G}:=\left(
\begin{array}{cccc}
1&1&0&0\\
0&0&1&1\\
0&0&1&1\\
1&1&0&0
\end{array}\right).
\]}
Then it is easy to see {(cf.~\cite[Theorem 4.3.3]{Lind_Marcus_1995})} that
\begin{equation}\label{eq:31}
h(X_{\mathcal G}){=\frac{{\log} \la(A_{\mathcal G})}{m}}=\frac{\log 2}{m},
\end{equation}
{where $\la(A_{\mathcal G})$ denotes the spectral radius of $A_{\mathcal G}$.}

\begin{proposition}
\label{p37}
Let $[p_L,p_R]\subseteq (q_{KL},M+1)$ be an entropy plateau of period $m$. Then for any $p\in[p_L,p_R)$,
\[
\dim_H(\mathscr{U}\cap [p,p_R])\geq \frac{\log 2}{m\log p_R}.
\]
\end{proposition}

(We will show in Section \ref{s4} that this holds in fact with equality.)

\begin{proof}
We will construct a sequence of subsets $\set{\Lambda_N}$ of $\us_{[p, p_R]}$ such that the Hausdorff dimension of $\pi_{p_R}(\Lambda_N)$ tends to $\frac{\log 2}{m\log p_R}$ as $N\ra\f$, where $\us_{[p, p_R]}:=\set{\al(\ell): \ell\in\overline{\ub}\cap[p, p_R]}$. This observation, when combined with Lemma \ref{l36} and the fact that the difference between $\ub$ and $\overline{\ub}$ is countable, will imply our lower bound. 

Let $a_1\dots a_m$ be the word such that $\al(p_L)=(a_1\ldots a_m)^\f$ and $\al(p_R)=a_1\ldots a_m^+(\overline{a_1\ldots a_m})^\f$.
Recall that $X_{\mathcal G}$ is a sofic subshift represented by the labeled graph $\mathcal G$ in Figure \ref{fig:1}.  For an integer $N\ge 2$ let $\Lambda_N$ be the set of sequences $(c_i)\in X_{\mathcal G}$ beginning with
  \[c_1\ldots c_{mN}=a_1\ldots a_m^+(\overline{a_1\ldots a_m})^{N-1}\]
   and the tail sequence $c_{mN+1}c_{mN+2}\ldots$ not containing  the word $a_1\ldots a_m^+(\overline{a_1\ldots a_m})^{N-1}$ or $\overline{a_1\ldots a_m^+}(a_1\ldots a_m)^{N-1}$. Note that since $\al(p){\prec}\al(p_R)$, we can choose $N$ large enough so that $\al(p){\prec}a_1\ldots a_m^+(\overline{a_1\ldots a_m})^{N-1}0^\f$.
  We claim that $\Lambda_N\subset\us_{[p, p_R]}$.

Observe that $a_1\ldots a_m^+(\overline{a_1\ldots a_m})^\f$ is the lexicographically largest sequence in $X_{\mathcal G}$, and $\overline{a_1\ldots a_m^+}(a_1\ldots a_m)^\f$ is the lexicographically smallest sequence in $X_{\mathcal G}$.  Take a sequence $(c_i)\in\Lambda_N$. Then $(c_i)$ has a prefix $a_1\ldots a_m^+(\overline{a_1\ldots a_m})^{N-1}$, and the tail $c_{mN+1}c_{mN+2}\ldots$ satisfies the inequalities
  \[\overline{(c_i)}\lle \overline{a_1\ldots a_m^+}({a_1\ldots a_m})^{N-1}M^\f \prec \si^n((c_i))\prec a_1\ldots a_m^+(\overline{a_1\ldots a_m})^{N-1}0^\f\lle (c_i)\] for all $n\ge mN$.
By Lemma \ref{l34}, to prove $(c_i)\in\us_{[p, p_R]}$ it suffices to prove $\overline{(c_i)}\prec \si^n((c_i))\prec (c_i)$ for all $1\le n<mN$. Note by Lemma \ref{l31}(i) that
 \begin{equation}\label{eq:32}
\overline{a_1\ldots a_{m-i}}\lle a_{i+1}\ldots a_m\prec a_1\ldots a_{m-i}\quad\textrm{for all}\quad 1\le i<m.
\end{equation}
This implies that
\[
a_{i+1}\ldots a_m^+\overline{a_1\ldots a_i}\lle a_1\ldots a_m\prec a_1\ldots a_m^+,
\]
and
\[
a_{i+1}\ldots a_m^+\succ a_{i+1}\ldots a_m\lge \overline{a_1\ldots a_{m-i}}
\]
for all $1\le i<m$. So $\overline{(c_i)}\prec \si^n((c_i))\prec (c_i)$ for all $1\le n<m$. Furthermore, by (\ref{eq:32}) it follows that
\begin{align*}
&\overline{a_1\ldots a_m^+}\prec \overline{a_1\ldots a_m}\lle a_{i+1}\ldots a_m a_1\ldots a_i\prec a_1\ldots a_m^+
\end{align*}
for all $0\le i<m$. Taking the reflection we obtain
\begin{equation}\label{eq:33}
\overline{a_1\ldots a_m^+}\prec \overline{a_{i+1}\ldots a_m a_1\ldots a_i}\prec a_1\ldots a_m^+
\end{equation}
for all $0\le i<m$. Since $c_{m(N-1)+1}\ldots c_{mN}=\overline{a_1\ldots a_m}$, we have  $c_{mN+1}\ldots c_{mN+m-1}=\overline{a_1\ldots a_{m-1}}$ (see Figure \ref{fig:1}). Then by (\ref{eq:33}) it follows that $\overline{(c_i)}\prec \si^n((c_i))\prec (c_i)$ for all $m\le n<mN$.
Therefore, $\overline{(c_i)}\prec \si^n((c_i))\prec (c_i)$ for all $n\ge 1$. So $(c_i)\in\us_{[p, p_R]}$, and hence $\Lambda_N\subset\us_{[p, p_R]}$.

Observe that $\pi_{p_R}(\Lambda_N)$ is the affine image of a graph-directed self-similar set whose Hausdorff dimension is arbitrarily close to the dimension of $\pi_{p_R}(X_{\mathcal G})$ as $N\ra\f$. Then
\[
\lim_{N\ra\f}\dim_H\pi_{p_R}(\Lambda_N)=\dim_H\pi_{p_R}(X_{\mathcal G})=\frac{\log 2}{m\log p_R}.
\]
Therefore, by the first inequality in Lemma \ref{l36} and the claim we conclude that
\begin{align*}
\dim_H(\overline{\ub}\cap[p, p_R])&\ge\dim_H\pi_{p_R}(\us_{[p, p_R]})\\
&\ge \lim_{N\ra\f}\dim_H\pi_{p_R}(\Lambda_N)=\frac{\log 2}{m\log p_R},
\end{align*}
completing the proof.
\end{proof}

Next, recall from (\ref{eq:25}) that $\vs_q$ is the set of sequences $(x_i)\in\Omega$ satisfying the inequalities:
 \[
 \overline{\al(q)}\lle \si^{ n}((x_i))\lle \al(q)\qquad\textrm{for all}\quad { n}\ge 0.
 \]

The next proposition shows that the set-valued map  $q\mapsto \vs_q$ does not vary too much inside an entropy plateau $[p_L, p_R]$, and gives a sharp estimate for the limit in Theorem \ref{main:1}(iv) when $q$ lies inside an entropy plateau.
 \begin{proposition}
\label{p38}
Let $[p_L, p_R]\subset(q_{KL}, M+1]$ be an entropy plateau of period $m$. Then
\begin{enumerate}[{\rm (i)}]
\item For all $p$ and $q$ with $p_L\leq p<q<p_R$,
\begin{equation}\label{eq:34}
\dim_H(\vs_q\backslash \vs_p)<\dim_H(\vs_{p_R}\backslash \vs_p)=\frac{\log 2}{ m}.
\end{equation}
 
\item For all $q\in(p_L,p_R]$,
\begin{equation}
\lim_{p\nearrow q}\dim_H(\overline{\ub}\cap(p,q))\le\frac{\log 2}{m\log q},
\label{eq:35}
\end{equation}
with equality if and only if $q=p_R$.
\end{enumerate}
\end{proposition}

\begin{proof}
First we prove (i).  By Lemma \ref{l31} there exists a word $a_1\ldots a_m$ such that
  \begin{equation} \label{eq:36}
  \al(p_L)=(a_1\ldots a_m)^\f\qquad\textrm{and}\qquad \al(p_R)=a_1\ldots a_m^+(\overline{a_1\ldots a_m})^\f.
  \end{equation}
  Take a sequence $(c_i)\in\vs_{p_R}\setminus\vs_{p_L}$.  Then there exists $j\ge 0$ such that
  \[c_{j+1}\ldots c_{j+m}=a_1\ldots a_m^+\qquad \textrm{or}\qquad  c_{j+1}\ldots c_{j+m}=\overline{a_1\ldots a_m^+}.\]
We claim that the tail sequence $c_{j+1}c_{j+2}\ldots\in X_{\mathcal G}$, where $X_{\mathcal G}$ is the sofic subshift determined by the labeled graph in Figure \ref{fig:1}.

By symmetry  we may  assume $c_{j+1}\ldots c_{j+m}=a_1\ldots a_m^+$. Since  $(c_i)\in\vs_{p_R}$, by (\ref{eq:36}) the sequence $(c_i)$ satisfies
\begin{equation}\label{eq:37}
\overline{a_1\ldots a_m^+}(a_1\ldots a_m)^\f\lle \si^{n}((c_i))\lle a_1\ldots a_m^+(\overline{a_1\ldots a_m})^\f
\end{equation}
for all ${n}\ge 0$. Taking  ${n}=j$ in  (\ref{eq:37})
it follows that $c_{j+m+1}\ldots c_{j+2m}\lle \overline{a_1\ldots a_m}$. Again, by (\ref{eq:37}) with ${n}=j+m$ we obtain that
$
  c_{j+m+1}\ldots c_{j+2m}\lge \overline{a_1\ldots a_m^+}.
$
So,  if $c_{j+1}\ldots c_{j+m}=a_1\ldots a_m^+$, then the next word $c_{j+m+1}\ldots c_{j+2m}$  has only two choices: it either equals $\overline{a_1\ldots a_m^+}$ or it equals $\overline{a_1\ldots a_m}$.
\begin{itemize}
\item If $c_{j+m+1}\ldots c_{j+2m}=\overline{a_1\ldots a_m^+}$, then by symmetry and using (\ref{eq:37}) it follows that the next word $c_{j+2m+1}\ldots c_{j+3m}$ equals either $a_1\ldots a_m$ or $a_1\ldots a_m^+$.

\item If $c_{j+m+1}\ldots c_{j+2m}=\overline{a_1\ldots a_m}$, then $c_{j+1}\ldots c_{j+2m}=a_1\ldots a_m^+\overline{a_1\ldots a_m}$. By using (\ref{eq:37}) with $k=j$ we have $c_{j+2m+1}\ldots c_{j+3m}\lle \overline{a_1\ldots a_m}$. Again, by (\ref{eq:37}) with $k=j+2m$ it follows that
the next word $c_{j+2m+1}\ldots c_{j+3m}$ equals either $\overline{a_1\ldots a_m^+}$ or  $\overline{a_1\ldots a_m}$.
\end{itemize}
By iteration of the above arguments we conclude that $c_{j+1}c_{j+2}\ldots \in X_{\mathcal G}$. This proves the claim: any sequence in $\vs_{p_R}\setminus\vs_{p_L}$ eventually ends with an element of $X_{\mathcal G}$.

Using the claim and  (\ref{eq:31}) it follows that
\begin{equation}\label{eq:38}
\dim_H(\vs_{p_R}\backslash \vs_p) \le \dim_H(\vs_{p_R}\setminus\vs_{p_L}) \le \dim_H X_{\mathcal G} \le \ { h(X_{\mathcal G})} =\frac{\log 2}{m}.
\end{equation}
On the other hand, since $p<p_R$ we have $\alpha(p){\prec}\alpha(p_R)=a_1\dots a_m^+(\overline{a_1\dots a_m})^\infty$, so there exists ${K}\in\N$ such that $\alpha(p){\prec}a_1\dots a_m^+(\overline{a_1\dots a_m})^{K} 0^\infty$. Hence, the follower set 
\[
F_{X_{\mathcal G}}\big(a_1\ldots a_m^+(\overline{a_1\dots a_m})^{ K}\big):=\set{(d_i)\in X_{\mathcal G}: d_1\ldots d_{m({ K}+1)}=a_1\ldots a_m^+(\overline{a_1\dots a_m})^{ K}} 
\]
is a subset of $\vs_{p_R}\setminus\vs_{p}$.
By (\ref{eq:31}) this implies that 
\begin{equation}\label{eq:39}
\dim_H(\vs_{p_R}\setminus\vs_{p})\ge \dim_H F_{X_{\mathcal G}}\big(a_1\ldots a_m^+(\overline{a_1\dots a_m})^{ K}\big)= { h(X_{\mathcal G})} =\frac{\log 2}{ m},
\end{equation}
where the first equality follows since, in view of the homogeneous structure of $X_{\mathcal G}$, there is no more efficient covering of this set than by cylinder sets of equal depth. Combining (\ref{eq:38}) and (\ref{eq:39}) gives 
\begin{equation}\label{eq:310}
\dim_H(\vs_{p_R}\setminus\vs_{p})=\frac{\log 2}{ m }.
\end{equation}

Next, observe that for $q\in(p_L, p_R)$ there exists $N\in\N$ such that
{\[
\al(q)\prec a_1\ldots a_m^+(\overline{a_1\ldots a_m})^N0^\f.
\] }
Then the words $a_1\ldots a_m^+(\overline{a_1\ldots a_m})^N$ and {$\overline{a_1\ldots a_m^+} (a_1\ldots a_m)^N$} are forbidden in $\vs_{q}$. By the above argument it follows that any sequence in $\vs_{q}\setminus\vs_{p}$ eventually ends with an element of 
{\begin{align}
\label{eq:311}
\begin{split}
X_{\mathcal{G}, N}&:=\big\{(d_i)\in X_{\mathcal G}: a_1\ldots a_m^+(\overline{a_1\ldots a_m})^N\textrm{ and }\\
& \qquad\qquad \overline{a_1\ldots a_m^+} (a_1\ldots a_m)^N\textrm{ do not occur in }(d_i)\big\}.
\end{split}
\end{align}}
{By (\ref{eq:31})} this implies that 
\begin{equation*}
\dim_H(\vs_q\setminus\vs_{p})\le\dim_H X_{\mathcal{G}, N}\le { h(X_{\mathcal G, N})} 
< { h(X_{\mathcal G})} =\frac{\log 2}{ m },
\end{equation*}
where the strict inequality holds by \cite[Corollary 4.4.9]{Lind_Marcus_1995}, since $X_{\mathcal G}$ is a transitive sofic subshift and $X_{\mathcal G, N}\subsetneq X_{\mathcal G}$. {Later in Lemma \ref{lem:XGN-entropy} we will give an explicit formula for $h(X_{\mathcal G, N})$.}
This completes the proof of (i).

To prove (ii), suppose first that $q\in(p_L, p_R)$. Let $a_1\dots a_m$ be the word such that \eqref{eq:36} holds. Take $p\in(p_L, q)\cap\overline{\ub}$. {By Lemma \ref{l21}}  it follows that for any $\ell\in(p, q)$ the quasi-greedy expansion $\al(\ell)$ begins with $a_1\ldots a_m^+$. As in the proof of (i), since $q<p_R$ it follows that there exists $N\in\N$ depending only on $q$ such that 
\[
\mathbf{U}_{(p,q)}:=\{\al(\ell): \ell\in\overline{\mathscr{U}}\cap(p,q)\}\subseteq X_{\mathcal{G},N},
\]
where $X_{\mathcal{G},N}$ was defined in \eqref{eq:311}. Therefore, by Lemma \ref{l36},
\begin{align*}
\lim_{p\nearrow q}\dim_H(\overline{\ub}\cap(p,q)) &\leq \lim_{p\nearrow q}\frac{h(\mathbf{U}_{(p,q)})}{\log p} \le \lim_{p\nearrow q}\frac{h(X_{\mathcal{G},N})}{\log p}\\
&=\frac{h(X_{\mathcal{G},N})}{\log q}<\frac{h(X_{\mathcal{G}})}{\log q}=\frac{\log 2}{m\log q}.
\end{align*}

For $q=p_R$ we have $h(\mathbf{U}_{(p,q)})\leq h(X_{\mathcal G})$, so as in the above calculation we obtain
\[
\lim_{p\nearrow p_R}\dim_H(\overline{\ub}\cap(p,p_R))\le\frac{\log 2}{m\log p_R}.
\]
The reverse inequality holds by Proposition \ref{p37}, and hence we have equality in \eqref{eq:35} for $q=p_R$.
\end{proof}

\begin{corollary}
\label{cor:39}
For any entropy plateau $[p_L, p_R]\subset(q_{KL}, M+1]$ and any $q\in(p_L,p_R]$,
\[
\dim_H(\vs_q\setminus\vs_{p_L})\le \dim_H\vs_{p_L},
\]
with equality if and only if $M=2k+1\geq 3$ and $q=p_R=q_\star(M)$.
\end{corollary}

\begin{proof}
Immediate from Lemma \ref{l31}(ii), Lemmas \ref{l28} and \ref{l29}, and Proposition \ref{p38}(i).
\end{proof}

As a final preparation for the proofs of Theorems \ref{main:1}, \ref{main-prime} and \ref{main:2}, we need the following results about the local dimension of the bifurcation sets $\overline{\bb}$ and $\overline{\ub}$. We first recall from \cite[Theorem 2]{Kalle-Kong-Li-Lv-2016} the local dimension of $\bb$.

  \begin{lemma}
  \label{l310}
  For any $q\in\overline{\bb}$ we have
  \[
  \lim_{\de\ra 0}\dim_H(\overline{\bb}\cap(q-\de, q+\de))=\dim_H\u_q.
  \]
  \end{lemma}

%Since the differences among $\bb^L, \bb^R$ and $\overline{\bb}$ are at most countable, Lemma \ref{l310} also holds for $\bb^L$ and $\bb^R$.

For the local dimension of $\ub$, we can prove the following:
 
\begin{lemma}\label{l311}
For any $q\in(1,M+1]$ we have 
\[\lim_{\de\ra 0}\dim_H(\overline{\ub}\cap(q-\de, q+\de))\le \dim_H\u_q.\]
\end{lemma}

\begin{proof}
Take $q\in(1,M+1]$. 
 By Lemmas \ref{l21}, \ref{l23} and \ref{l34} it follows that for each $\ell\in\overline{\ub}\cap(q-\de, q+\de)$ the quasi-greedy expansion $\al(\ell)$ belongs to $\us_{q+\de}$, where we set $\us_{q+\de}=\Omega$ if $q+\de>M+1$. In other words,  using the notation of Lemma \ref{l36}, 
 \[\us_{(q-\de,q+\de)}\subseteq\us_{q+\de}.\]
  We now obtain by Lemma \ref{l36} and Lemma \ref{l25},
\begin{align*}
\dim_H\big(\overline{\mathscr{U}}\cap(q-\de,q+\de)\big) &\leq \frac{h(\mathbf{U}_{(q-\de,q+\de)})}{\log(q-\de)} \leq \frac{h(\mathbf{U}_{q+\de})}{\log(q-\de)}\\
&{\le}\ \frac{\log(q+\de)}{\log(q-\de)}\dim_H \pazocal{U}_{q+\delta} \to \dim_H \pazocal{U}_q
\end{align*}
as $\de \to 0$. This completes the proof.
\end{proof}

We are now ready to prove Theorems \ref{main:1}, \ref{main-prime} and \ref{main:2}.		
		
\begin{proof}[Proof of Theorem \ref{main:1}] 
Suppose $M=1$ or $M$ is even. We prove (i) $\Leftrightarrow$ (ii) and (i) $\Rightarrow$ (iii) $\Rightarrow$ (iv) $\Rightarrow$ (i).

First we prove (i) $\Rightarrow$ (ii). Let $q\in\bb^L$, and take $p\in(1, q)$.  
 Then $H(p)< H(q)$ by the definition of $\bb^L$, so Lemma \ref{l28} implies
  \[
  \dim_H\us_p= { H(p)} <  { H(q)} =\dim_H\us_q.
  \]
Therefore,
  \[
  \dim_H(\us_q\setminus\us_p)=\dim_H\us_q>\dim_H\us_p\ge 0.
  \]

Next, we prove (ii) $\Rightarrow$ (i). Let $q\in(1,M+1]\setminus\bb^L$.  By (\ref{e16}) we have $q\in(1, q_{KL}]$ or $q\in(p_L, p_R]$ for some entropy plateau $[p_L, p_R]\subset(q_{KL}, M+1]$.  If $q\in(1,q_{KL}]$, then by Lemma \ref{l25} we have
  \[
  \dim_H(\us_q\setminus\us_p)=\dim_H\us_q=0
  \]
  for any $p\in(1, q)$.
  Suppose $q\in(p_L, p_R]\subset(q_{KL}, M+1]$, and take $p\in(p_L, q)$. By Corollary \ref{cor:39} and Lemma \ref{l29} it follows that
  \begin{align*}
  \dim_H(\us_q\setminus\us_p) \le \dim_H(\us_{q}\setminus\us_{p_L})
  &=\dim_H(\vs_{q}\setminus\vs_{p_L})\\
  &<\dim_H\vs_{p_L}=\dim_H\us_{p_L}\le \dim_H\us_q.
  \end{align*}
Thus, (ii) $\Rightarrow$ (i).

We next prove (i) $\Rightarrow$ (iii). Take $q\in\bb^L$. Then $q>q_{KL}$ by (\ref{e16}), so Lemma \ref{l25} yields $\dim_H\u_q>0$. Thus, it remains to prove that $\lim_{p\nearrow q}\dim_H(\bb\cap(p, q))=\dim_H\u_q$. Since $\bb\subset\ub$, by Lemma \ref{l311} it suffices to prove
  \begin{equation}\label{eq:312}
  \lim_{p\nearrow q}\dim_H(\bb\cap(p, q))\ge \dim_H\u_q.
  \end{equation}

  Fix $\ep>0$. By Lemma \ref{l25} the function $q\mapsto \dim_H\u_q$ is continuous, so there exists $p_0:=p_0(\ep)\in(1, q)$ such that
    \begin{equation}
  \label{eq:313}
 \dim_H\u_{p}\ge \dim_H\u_q-\ep\qquad\textrm{for all}\quad p \in(p_0, q).
  \end{equation}
  Since $q\in\bb^L$, by the topological structure of the bifurcation set $\bb^L$ there exists a sequence of entropy plateaus $\set{[p_L(n), p_R(n)]}$ such that $p_L(n)\nearrow q$ as $n\ra\f$. Fix $p\in(p_0,q)$. Then there exists a large integer $N$ such that $p_L(N)\in(p, q)$. Observe that $p_L(N)\in\bb^L\subset\overline{\bb}$ and the difference $\overline{\bb}\setminus\bb$ is countable. By Lemma \ref{l310} there exists $\de>0$ such that
  \begin{equation}
	\label{eq:314}
 (p_L(N)-\de, p_L(N)+\de) \subseteq (p, q),
	\end{equation}
	and
	\begin{equation}
	\label{eq:315}
	\dim_H\big(\bb\cap(p_L(N)-\de, p_L(N)+\de)\big) \ge \dim_H\u_{p_L(N)}-\ep.
	\end{equation}
 By (\ref{eq:313}), \eqref{eq:314} and (\ref{eq:315}) it follows that
  \begin{align*}
  \dim_H(\bb\cap(p, q))&\ge\dim_H\big(\bb\cap(p_L(N)-\de, p_L(N)+\de)\big)\\
  &\ge\dim_H\u_{p_L(N)}-\ep\ge \dim_H\u_q-2\ep.
  \end{align*}
Since this holds for all $p\in(p_0(\ep),q)$, we obtain (\ref{eq:312}).
This proves (i) $\Rightarrow$ (iii).

Note that (iii) $\Rightarrow$ (iv) follows directly from Lemma \ref{l311} {since $\bb\subset\ub$}. 

It remains to prove (iv) $\Rightarrow$ (i). Let $q\in(1,M+1]\setminus\bb^L$.  By (\ref{e16}) it follows that $q\in(1, q_{KL}]$ or $q\in(p_L, p_R]$ for some entropy plateau $[p_L, p_R]\subset(q_{KL}, M+1]$.  If $q\in(1, q_{KL}]$, then $\dim_H\u_q=0$. Now we consider  $q\in(p_L, p_R]\subset(q_{KL}, M+1]$. If $q\notin\overline{\ub}$, then $\lim_{p\nearrow q}\dim_H(\ub\cap(p, q))=0$.
So let $q\in\overline{\ub}\cap(p_L, p_R]$. If $q<p_R$, then Proposition \ref{p38}(ii), Lemma \ref{l31}(ii) and Lemma \ref{l25} give
\begin{equation}
\label{eq:316}
\lim_{p\nearrow q}\dim_H(\overline{\ub}\cap(p, q))<\frac{\log 2}{m\log q}\leq \frac{h(\us_{p_L})}{\log q}=\frac{h(\us_q)}{\log q}=\dim_H \u_q.
\end{equation}
Similarly, if $q=p_R$, then Lemma \ref{l31}(ii) holds with strict inequality, and we obtain the same end result as in \eqref{eq:316}, but with the first inequality replaced by ``$\leq$" and the second inequality replaced by ``$<$".	This proves (iv) $\Rightarrow$ (i), and completes the proof of Theorem \ref{main:1}.
\end{proof}	
	
\begin{proof}[Proof of Theorem \ref{main-prime}]
The proof of Theorem \ref{main-prime} is, for the most part, the same as the proof of Theorem \ref{main:1}. Asssume $M=2k+1\geq 3$. We need only check the following two facts for the entropy plateau $[p_L,p_R]=[k+2,q_\star]$, where $q_\star=q_\star(M)$:
\begin{equation}
\label{eq:317}
\dim_H(\us_{q_\star}\backslash \us_p)=\dim_H(\us_{q_\star}) \qquad\mbox{for any $p\in(1,q_\star)$},
\end{equation}
and
\begin{equation}
\label{eq:318}
\lim_{p\nearrow q_\star}\dim_H(\overline{\ub}\cap(p, q_\star))=\dim_H \u_{q_\star}.
\end{equation}
Here \eqref{eq:317} is clear for $p\in(1, {k+2})$, since $\dim_H \us_p<\dim_H \us_{q_\star}$. For $p\in[{k+2}, q_\star)$, \eqref{eq:317} follows from Proposition \ref{p38}(i) and the equality statement in Lemma \ref{l31}(ii), noting that {$[k+2, q_\star]$ is an entropy plateau of} period $m=1$.

Similarly, \eqref{eq:318} follows from the equality statements in Proposition \ref{p38}(ii) and Lemma \ref{l31}(ii).
\end{proof}

\begin{proof}[Proof of Theorem \ref{main:2}]
The proofs of (i) $\Rightarrow$ (ii) and (iii) $\Rightarrow$ (iv) are completely analogous to the proofs of the corresponding implications in Theorem \ref{main:1}. 

Consider the implication (ii) $\Rightarrow$ (i). Suppose $q\in(1,M+1]\setminus\bb^R$. By (\ref{e16}) we have $q\in(1, q_{KL})$ or $q\in[p_L, p_R)$ for some entropy plateau $[p_L, p_R]\subset(q_{KL}, M+1]$. A similar argument as in the proof of Theorem \ref{main:1} shows that either $\dim_H\us_q=0$ for $q\in(1,q_{KL})$, or $\dim_H(\us_r\setminus\us_q)<\dim_H\us_r$ for any $r\in(q, p_R)$. This proves (ii) $\Rightarrow$ (i).

Next, consider the implication (i) $\Rightarrow$ (iii). Take $q\in\bb^R$. Then $q\ge q_{KL}$. If $q\ne q_{KL}$, then by Lemma \ref{l25} we have $\dim_H\u_q>0$.  Since $q\in\bb^R$, there exists a sequence of entropy plateaus $\set{[\tilde p_L(n), \tilde p_R(n)]}$ such that $\tilde p_L(n)\searrow q$ as $n\ra\f$. Using the continuity of the function $q\mapsto \dim_H\u_q$ and Lemma \ref{l310},  we can show as in the proof of Theorem \ref{main:1} that
$\lim_{r\searrow q}\dim_H(\bb\cap(q, r))=\dim_H\u_q.$
This proves (i) $\Rightarrow$ (iii).

Finally, consider the implication (iv) $\Rightarrow$ (i).
For $q\in(1,M+1]\setminus\bb^R$ we have $q\in(1,q_{KL})$ or $q\in[p_L, p_R)$ for some entropy plateau $[p_L, p_R]\subset(q_{KL},M+1]$. By the same argument as in the proof of Theorem \ref{main:1} we can prove that either $\dim_H\u_q=0$ for $q< q_{KL}$, or $\lim_{r\searrow q}\dim_H(\overline{\ub}\cap(q, r))<\dim_H\u_q$ for $q\in[p_L, p_R)$. This establishes (iv) $\Rightarrow$ (i).
\end{proof}

\section{Hausdorff dimension of $\ub\setminus\bb$} \label{s4}

In this section we will calculate the Hausdorff dimension of the difference set $\ub\setminus\bb$ and prove Theorem \ref{main:4}. First, we prove the following result for the local dimension of $\ub$ inside any entropy plateau $[p_L, p_R]$.

\begin{theorem}
\label{thm:41}
Let $[p_L, p_R]\subset(q_{KL}, M+1)$ be an entropy plateau of period $m$. Then
\[\dim_H(\ub\cap[p_L, p_R])=\frac{\log 2}{m\log p_R}.\]
\end{theorem}

Observe that the lower bound in Theorem \ref{thm:41}, that is, the inequality 
$$\dim_H(\ub\cap[p_L, p_R])\ge\frac{\log 2}{m\log p_R},$$
follows from Proposition \ref{p37} by setting $p=p_L$. The proof of the reverse inequality is more tedious, and we will give it in several steps.

Observe that $\inf\ub=q_{KL}$, and any entropy plateau $[p_L, p_R]\subset(q_{KL}, M+1]$ satisfies $\al(q_{KL})\prec \al(p_L)\prec \al(M+1)$. In the following we fix an arbitrary entropy plateau $[p_L, p_R]\subset(q_{KL}, M+1]$ of period $m$ such that $\al(p_L)=(a_1\ldots a_m)^\f$  and $\al(p_R)=a_1\ldots a_m^+(\overline{a_1\ldots a_m})^\f$. 
Recall the definition of the generalized Thue-Morse sequence $(\lambda_i)=(\lambda_i(M))$ from \eqref{eq:22}, which has the property that $\al(q_{KL})=(\lambda_i)$. If $M=1$, then
  \[
  1101\ldots=\lambda_1\lambda_2\ldots\prec (a_1\ldots a_m)^\f \prec 1^\f,
  \]
so $m\ge 3$. Similarly, if $M=2$, we have
\[
210 201 \ldots=\lambda_1\lambda_2\ldots \prec (a_1\ldots a_m)^\f \prec 2^\f,
\]
so $m\ge 2$. But when $M\geq 3$, it is possible to have $m=1$. In short, we have the inequality
\begin{equation}
\label{eq:M-plus-m}
M+m\geq 4.
\end{equation}

We divide the interval $(p_L,p_R)$ into a sequence of smaller subintervals by defining a sequence of bases $\set{q_n}_{n=1}^\f$ in $(p_L,p_R)$. Let $\hat{q}=\min(\overline{\ub}\cap(p_L,p_R))$, and for $n\ge 1$ let $q_n\in(p_L,p_R)$ be defined by
\begin{equation}\label{e43}
\al(q_n)=\left(a_1\dots a_m^+(\overline{a_1\dots a_m})^{n-1}\overline{a_1\dots a_m^+}\right)^\f.
\end{equation}
Note that $\hat{q}$ is a de Vries-Komornik number which has a Thue-Morse type quasi-greedy expansion
\begin{equation}\label{e44}
\al(\hat{q})=a_1\dots a_m^+\ \ \overline{a_1\dots a_m}\ \ \overline{a_1\dots a_m^+}\,a_1\dots a_m^+\cdots 
\end{equation}
That is, $\al(\hat{q})$ is the sequence $\al_1\al_2\dots$ given by $\al_1\dots \al_m=a_1\dots a_m^+$, and recursively, for $i\geq 0$, $\al_{2^i m+1}\dots \al_{2^{i+1}m}=\overline{\al_1\dots\al_{2^i m}\,}^+$.
Then  $\al(q_1)\prec \al(\hat{q})\prec \al(q_2)\prec\cdots\prec \al(p_R)$, and $\al(q_n)\nearrow \al(p_R)$ as $n\ra\f$. By Lemma \ref{l21} it follows that
\[
q_1<\hat{q}<q_2<q_3<\cdots<p_R, \quad \textrm{and}\quad q_n\nearrow p_R \quad\textrm{as}\ n\ra\f.
\]

{We will bound the dimension of $\overline{\ub}\cap[q_n, q_{n+1}]$ for each $n\in\N$. In preparation for this, we first determine the entropy of the subshift $X_{\mathcal{G},N}$ defined in \eqref{eq:311}.

\begin{lemma} \label{lem:XGN-entropy}
The topological entropy of $X_{\mathcal{G},N}$ is given by
\[
h(X_{\mathcal{G},N})=\frac{\log \varphi_N}{m},
\]
where $\varphi_N$ is the unique root in $(1,2)$ of $1+x+\dots+x^{N-1}=x^N$.
\end{lemma}

\begin{proof}
The $m$-block map $\Phi$ defined by 
{\[\Phi(a_1\dots a_m^+)=\Phi\big(\overline{a_1\dots a_m^+}\big)=1,\quad \Phi(a_1\dots a_m)=\Phi(\overline{a_1\dots a_m})=0\]}
 induces a two-to-one map $\phi$ from $X_{\mathcal{G},N}$ into $\{0,1\}^\N$. {Recall that $X_{\mathcal G, N}$ is the subset of $X_{\mathcal G}$ with forbidden blocks $a_1\ldots a_m^+(\overline{a_1\ldots a_m})^N$ and $\overline{a_1\ldots a_m^+}(a_1\ldots a_m)^N$.} Then $Y:=\phi(X_{\mathcal{G},N})$ is the  subshift of finite type in $\{0,1\}^\N$ of sequences avoiding the word {$10^N$}. It is well known that $h(Y)=\log\varphi_N$ (cf.~\cite[Exercise 4.3.7]{Lind_Marcus_1995}); hence, $h(X_{\mathcal{G},N})=(\log \varphi_N)/m$.
\end{proof}

\begin{lemma}\label{lem:44}
For any $n\ge 1$, we have
\[
\dim_H(\overline{\ub}\cap[q_n, q_{n+1}])\le \frac{\log \varphi_{n+1}}{m\log q_n}.
\]
\end{lemma}

\begin{proof}
Fix $n\geq 1$. Note by (\ref{e43}) and (\ref{e44}) that for any $p\in\overline{\ub}\cap[q_n,q_{n+1}]$, $\al(p)$ begins with $a_1\dots a_m^+$, and $\al(p)\in\vs_p \subseteq\vs_{q_{n+1}}$. By a similar argument as in the proof of Proposition \ref{p38} it follows that $\al(p)\in X_{\mathcal G}$, and $\al(p)$ does not contain the subwords {$a_1\ldots a_m^+(\overline{a_1\dots a_m})^{n+1}$ and $\overline{a_1\ldots a_m^+}(a_1\dots a_m)^{n+1}$}, where $X_{\mathcal G}$ is the sofic subshift represented by the labeled graph $\mathcal G=(G, \mathcal L)$ in Figure \ref{fig:1}. In other words, $\al(p)\in X_{\mathcal{G},n+1}$. By Lemma \ref{lem:XGN-entropy} this implies
\begin{equation}
\label{eq:entropy-rough-estimate}
h(\us_{[q_n, q_{n+1}]})\le h(X_{\mathcal{G},n+1})=\frac{\log \varphi_{n+1}}{m}.
\end{equation}
Applying Lemma \ref{l36} with $I=[q_n,q_{n+1}]$ completes the proof.
\end{proof}

}

The next step is to prove that the upper bound in Lemma \ref{lem:44} is smaller than {$\log 2/(m\log p_R)$}. This requires us to show that $q_n$ is sufficiently close to $p_R$, which we accomplish by applying a \emph{transversality} technique (see \cite{Pol-Sim-95,Sol-95}) to certain polynomials associated with $q_n$ and $p_R$. For this we need the estimation of the Komornik-Loreti constants $q_{KL}(M)$. Recall from Example \ref{ex:24} that
\[
q_{KL}(1)\approx 1.78723,\qquad q_{KL}(2)\approx 2.53595\qquad\textrm{and}\qquad q_{KL}(3)\approx 2.91002.
\]
We emphasize that $q_{KL}(M)\ge (M+2)/2$ for each $M\ge 1$, and  the map  $M\mapsto q_{KL}(M)$ is strictly increasing.

\begin{lemma}
\label{lem:polynomials}
Let {$[p_L, p_R]\subset(q_{KL}, M+1]$ be an entropy plateau such that $\al(p_L)=(a_1\ldots a_m)^\f$ and $\al(p_R)=a_1\ldots a_m^+(\overline{a_1\ldots a_m})^\f$.} Define the polynomials
\begin{align}
\begin{split}
P(x)&:=a_1 x+\dots +a_{m-1}x^{m-1}+(1+a_m^+)x^m\\
& \qquad\quad +(\overline{a_1}-a_1)x^{m+1}+\dots+(\overline{a_{m-1}}-a_{m-1})x^{2m-1}+(\overline{a_m}-a_m^+)x^{2m}-1
\end{split}
\label{eq:P}
\end{align}
and
\begin{equation}
\label{eq:Q}
Q_n(x):=P(x)-x^{m(n+1)}(\overline{a_1} x+\dots+\overline{a_m} x^m), \qquad n\in\N.
\end{equation} 
\begin{enumerate}[{\rm (i)}]
\item The number $1/p_R$ is the unique zero of $P$ in $[1/(M+1),1]$.
\item The number $1/q_n$ is the unique zero of $Q_n$ in $[1/(M+1),1]$, for all $n\in\N$.
\item $P'(x)\geq a_1$ for all $x\in [1/p_R,1/p_L]$.
\end{enumerate}
\end{lemma}

\begin{proof}
(i) Since $\al(p_R)=a_1\dots a_m^+(\overline{a_1\dots a_m})^\infty$, it follows that $1/p_R$ is the unique solution in $[1/(M+1),1]$ of
\begin{align*}
1 &= a_1 x+a_2 x^2+\dots+a_{m-1}x^{m-1}+a_m^+x^m\\
& \qquad\qquad +x^m(\overline{a_1} x+\dots+\overline{a_m} x^m)+x^{2m}(\overline{a_1} x+\dots+\overline{a_m} x^m)+\dots\\
&= a_1 x+a_2 x^2+\dots+a_{m-1}x^{m-1}+a_m^+x^m+\frac{x^m(\overline{a_1} x+\dots+\overline{a_m} x^m)}{1-x^m}.
\end{align*}
Expanding and rearranging terms we see that $1/p_R$ is the unique zero in $[1/(M+1),1]$ of $P$.

(ii) By \eqref{e43}, it follows that the greedy expansion of $1$ in base $q_n$ is
\[
\beta(q_n)=a_1\dots a_m^+(\overline{a_1\dots a_m})^n 0^\infty,
\]
so $1/q_n$ is the unique root in $[1/(M+1),1]$ of the equation
\[
1=a_1 x+\dots+a_{m-1}x^{m-1}+a_m^+x^m+\frac{x^m(\overline{a_1} x+\dots+\overline{a_m} x^m)(1-x^{mn})}{1-x^m}.
\]
Expanding and rearranging gives that $1/q_n$ is the unique zero in $[1/(M+1),1]$ of $Q_n$.

(iii) Consider first the case $m=1$. In this case, the polynomial $P$ should be interpreted as 
\[
P(x)=(1+a_1^+)x+(\overline{a_1}-a_1^+)x^2-1.
\]
Now observe that, since $\al(p_L)=a_1^\f$, it follows that $p_L=a_1+1$. So for $x\in[1/p_R,1/p_L]$, we have in particular that $x\leq 1/(a_1+1)$. Therefore, since $a_1\geq (M+1)/2$,
\begin{align*}
P'(x)&=1+a_1^++2(\overline{a_1}-a_1^+)x=2+a_1+2(M-2a_1-1)x\\
&\geq 2+a_1+\frac{2(M-2a_1-1)}{a_1+1}=a_1+\frac{2(M+1)}{a_1+1}-2\\
&\geq a_1,
\end{align*}
where the last inequality follows since $a_1\leq M$.

Assume next that $m\geq 2$. Here we use that the greedy expansion of $1$ in base $p_L$ is $\beta(p_L)=a_1\dots a_m^+0^\f$, so 
\begin{equation}
\label{eq:pL-equation}
a_1 p_L^{-1}+\dots+a_{m-1}p_L^{-(m-1)}+a_m^+p_L^{-m}=1. 
\end{equation}
Hence,
\begin{equation}
\label{eq:poly-inequality}
a_1 x+\dots+a_{m-1}x^{m-1}+a_m^+x^m\leq 1 \qquad\mbox{for}\ 0\leq x\leq 1/p_L.
\end{equation}
Now for $0\leq x\leq 1/p_L$, writing $\overline{a_k}-a_k$ as $M-2a_k$, we have
\begin{align*}
P'(x)&=a_1+\sum_{k=2}^{m-1}ka_k x^{k-1}+m(1+a_m^+)x^{m-1}\\
&\qquad\qquad +\sum_{k=1}^{m-1}(m+k)(M-2a_k)x^{m+k-1}+2m(M-2a_m^++1)x^{2m-1}\\
&\geq a_1+\sum_{k=2}^{m-1}\left\{ka_k x^{k-1}+\big(M(m+k)-2(k-1)a_k\big)x^{m+k-1}\right\}\\
&\qquad\qquad +\left\{m(1+a_m^+)-2(m+1)\right\}x^{m-1}+Mx^m\{m+1+2mx^{m-1}\}\\
&\qquad\qquad +2\{m-(m-1)a_m^+\}x^{2m-1},
\end{align*}
where the inequality follows by multiplying both sides of \eqref{eq:poly-inequality} by $m+1$ and some algebraic manipulation. Here, the terms in the summation over $k=2,\dots,m-1$ are positive, since $a_k\leq M$ and so $M(m+k)-2(k-1)a_k\geq M(m-k+2)>0$. The sum of the remaining terms is increasing in $a_m^+$, since the coefficient of $a_m^+$ is 
$$mx^{m-1}-2(m-1)x^{2m-1}\geq mx^{m-1}(1-2x^m)\geq 0,$$ 
using that $m\geq 2$ and $x\leq 1/p_L\leq {1/q_{KL}(1)}\leq 0.6$, which holds for all $M\geq 1$. Since $a_m^+\geq 1$, it follows that
\begin{align*}
P'(x)&\geq a_1-2x^{m-1}+Mx^m\{m+1+2mx^{m-1}\}+2x^{2m-1}\\
&\geq a_1{-2x^{m-1}+M(m+1)x^m}=a_1+x^{m-1}\{M(m+1)x-2\}.
\end{align*}
At this point, we need that $x\geq 1/p_R\geq 1/(M+1)$. When $M\geq 2$, this implies
\[
M(m+1)x-2\geq 3Mx-2\geq \frac{3M}{M+1}-2=\frac{M-2}{M+1}\geq 0,
\]
recalling our assumption that $m\geq 2$. When $M=1$, we have $m\geq 3$ by \eqref{eq:M-plus-m}, and so $M(m+1)x-2\geq 4x-2\geq 0$, since $x\geq 1/2$. In both cases, it follows that $P'(x)\geq a_1$.
\end{proof}

The following elementary lemma (an easy consequence of the mean value theorem) is the key to the proof of the next inequality, in Lemma \ref{lem:46} below.

\begin{lemma} \label{lem:45}
Let $f:\R\to\R$ be a continuously differentiable function which has a zero  $x_0$, and let $\gamma>0$, $\delta>0$. Suppose $|f'(x)|\geq \gamma$ for all $x\in(x_0-\delta,x_0+\delta)$. If $g$ is a continuous function such that
\[
|g(x)-f(x)|\leq \gamma\delta \quad \mbox{for all} \quad x\in(x_0-\delta,x_0+\delta),
\]
then $g$ has at least one zero in $[x_0-\delta,x_0+\delta]$.
\end{lemma}

\begin{lemma} \label{lem:46}
{For each $n\geq 1$,
\[
\frac{\log \varphi_{n+1}}{\log 2}<\frac{\log q_n}{\log p_R}.
\]
}
\end{lemma}

\begin{proof}
Set $\mu_n:=1/q_n$ for $n\geq 1$, and set $\mu^*:=1/p_R$. {Then $\mu_n>\mu^*$ for all $n\ge 1$.} We will use Lemma \ref{lem:45} to show that $\mu_n$ is sufficiently close to $\mu^*$. 

By Lemma \ref{lem:polynomials}, $\mu^*$ is the unique zero in $[1/(M+1),1]$ of the polynomial $P(x)$ from \eqref{eq:P}, and $\mu_n$ is the unique zero in $[1/(M+1),1]$ of the polynomial $Q_n(x)$ from \eqref{eq:Q}. Moreover,
\begin{equation}
\label{eq:slope-lower-bound}
P'(x)\geq a_1\geq \frac{M+1}{2} \qquad \textrm{for all}\quad \mu^*\leq x\leq 1/p_L.
\end{equation}
In order to estimate the difference $P(x)-Q_n(x)$, we show first that
\begin{equation}
\label{eq:poly-estimate}
\overline{a_1}x+\dots+\overline{a_m}x^m<1 \qquad \textrm{for all}\quad 0\leq x\leq 1/p_L.
\end{equation}
Observe that
\[
\overline{a_1}x+\dots+\overline{a_m}x^m=\frac{Mx(1-x^m)}{1-x}-(a_1 x+\dots+a_m x^m).
\]
Hence, recalling \eqref{eq:pL-equation}, we have for $0\leq x\leq 1/p_L$,
\begin{align*}
\overline{a_1}x+\dots+\overline{a_m}x^m \leq \overline{a_1}p_L^{-1}+\dots+\overline{a_m}p_L^{-m} &=\frac{M(1-p_L^{-m})}{p_L-1}-(1-p_L^{-m})\\
&=(1-p_L^{-m})\left(\frac{M}{p_L-1}-1\right)\\
&\leq 1-p_L^{-m}<1,
\end{align*}
where the next-to-last inequality follows since $p_L\geq q_{KL}{(M)}\geq (M+2)/2$.  
This proves \eqref{eq:poly-estimate}.

Recall our convention that logarithms are taken with respect to base $M+1$. Below, we write $\ln x$ for the natural logarithm of $x$.
Suppose we can show, for some number $\delta_n>0$, that
\begin{equation}
\label{eq:delta-bound}
\mu_n-\mu^* \leq \delta_n.
\end{equation}
Using the inequality $\ln(1+x)\le x$ for any $x>-1$, it then follows that
\[
\ln\mu_n-\ln\mu^*=\ln\left(1+\frac{\mu_n-\mu^*}{\mu^*}\right)\leq\frac{\mu_n-\mu^*}{\mu^*}\leq\frac{\delta_n}{\mu^*}=\delta_n p_R,
\]
and so
\begin{equation}
\label{e48}
\frac{\ln q_n}{\ln p_R}=1+\frac{\ln q_n-\ln p_R}{\ln p_R}=1-\frac{\ln\mu_n-\ln\mu^*}{\ln p_R}\geq 1-\frac{\delta_n p_R}{\ln p_R}.
\end{equation}

{Next, observe that {$\varphi_{n+1}^{n+1}(1-\varphi_{n+1})=1-\varphi_{n+1}^{n+1}$, whence $\varphi_{n+1}^{n+1}(2-\varphi_{n+1})=1$}. It follows that
\[
2-\varphi_{n+1}=\varphi_{n+1}^{-(n+1)}>2^{-(n+1)},
\]
and hence,
\[
\ln \varphi_{n+1}-\ln 2=\ln\left(1+\frac{\varphi_{n+1}-2}{2}\right)\leq \frac{\varphi_{n+1}-2}{2}<-\frac{1}{2^{n+2}}.
\]
This gives
\begin{equation}
\label{e49}
\frac{\ln\varphi_{n+1}}{\ln 2}<1-\frac{1}{2^{n+2}\ln 2}.
\end{equation}

In view of \eqref{e48} and \eqref{e49} and the change-of-base formula $\ln x=\ln(M+1)\cdot \log x$, it then remains to show that
\begin{equation}
\label{e410}
\frac{\delta_n p_R}{\log p_R}<\frac{1}{2^{n+2}\log 2}\qquad{\textrm{for each}\quad n\ge 1}.
\end{equation}

By \eqref{eq:poly-estimate} and (\ref{eq:Q}) we have
\begin{equation*}
0\leq P(x)-Q_n(x)\leq p_L^{-m(n+1)}\leq q_{KL}^{-m(n+1)}, \qquad x\in[0,1/p_L]. 
\end{equation*}
Since we know that $\mu_n\in[\mu^*,1/p_L]$ and moreover, $\mu_n$ is the unique root of $Q_n$ in $[1/(M+1),1]$, it follows from \eqref{eq:slope-lower-bound} and Lemma \ref{lem:45} (with $\gamma=(M+1)/2$) that \eqref{eq:delta-bound} holds with 
\[
\delta_n=\frac{2}{M+1}q_{KL}^{-m(n+1)}.
\]

(i) Assume first that $m\geq 2$. Then we can estimate
\begin{align}
\begin{split}
\left(2^{n+2}\log 2\right)\frac{\delta_n p_R}{\log p_R} &\leq 2\log 2 \cdot \frac{2}{M+1}\cdot \frac{M+1}{\log q_{KL}}\left(\frac{2}{q_{KL}^m}\right)^{n+1}\\
&=\frac{4\log 2}{\log q_{KL}}\left(\frac{2}{q_{KL}^m}\right)^{n+1},
\end{split}
\label{eq:fraction-estimate}
\end{align}
where the inequality follows since $p_R\leq M+1$ and $\log p_R\geq\log q_{KL}$. Now observe that $\log 2/\log q_{KL}\leq {\log 2/\log q_{KL}(1)\le }\log 2/\log 1.787<1.2$. Furthermore, if $M\geq 2$ then $2/q_{KL}^m{\le 2/(q_{KL}(2))^2\le 2/(2.5)^2}<0.33$; and if $M=1$, then $m\geq 3$ by \eqref{eq:M-plus-m} and so $2/q_{KL}^m\leq 2/(1.787)^3<0.36$. In both cases, it follows that
\[
\left(2^{n+2}\log 2\right)\frac{\delta_n p_R}{\log p_R}\leq (4.8)(0.36)^{n+1}\leq (4.8)(0.36)^2<1,
\]
for all $n\geq 1$. Thus, we have proved \eqref{e410} in the case $m\geq 2$.

(ii) Assume next that $m=1$, so $M\geq 3$ by \eqref{eq:M-plus-m}. In this case, the bound in \eqref{eq:fraction-estimate} is just too large for $n=1$. But we can use the easily verified fact that the function $x\mapsto x/\log x$ is increasing on $[e,\infty)$ and $p_R\geq q_{KL}(3)\ge 2.9>e$, to replace the factor $\log q_{KL}$ in \eqref{eq:fraction-estimate} with the sharper $\log(M+1)$. Since $\log(M+1)\geq\log 4=2\log 2$, this gives the estimate
\begin{align*}
\left(2^{n+2}\log 2\right)\frac{\delta_n p_R}{\log p_R} &\leq 2\log 2 \cdot \frac{2}{M+1}\cdot \frac{M+1}{\log(M+1)}\left(\frac{2}{q_{KL}}\right)^{n+1}\\
&\leq 2\left(\frac{2}{q_{KL}}\right)^2\leq 2\left(\frac{2}{2.9}\right)^2\approx .9512<1.
\end{align*}
}

In both cases above, we have found a $\delta_n$ such that \eqref{eq:delta-bound} holds, and proved \eqref{e410}. Therefore, the proof of the Lemma is complete.
\end{proof}

\begin{proof}[Proof of the upper bound in Theorem \ref{thm:41}]
By Lemmas \ref{lem:44} and \ref{lem:46}, we have
\[
\dim_H(\overline{\ub}\cap[q_n, q_{n+1}]) {<} \frac{\log 2}{m\log p_R} \qquad\mbox{for each $n\geq 1$}.
\]
Since $\overline{\ub}\cap(p_L,p_R) \subseteq \bigcup_{n=1}^\infty (\overline{\ub}\cap[q_n, q_{n+1}])$, it follows from the countable stability of Hausdorff dimension that
\[
\dim_H(\overline{\ub}\cap[p_L,p_R]) \le \sup_{n\ge 1} \dim_H(\overline{\ub}\cap[q_n, q_{n+1}]) \le \frac{\log 2}{m\log p_R},
\]
establishing the upper bound.
\end{proof}

\begin{remark}\label{rem:47}
{
The above method of proof shows that in fact, for any $\ep>0$ we have $\dim_H(\overline{\ub}\cap[p_L,p_R-\ep])< \dim_H(\overline{\ub}\cap[p_L,p_R])$ and therefore, 
\[\dim_H(\overline{\ub}\cap[p_R-\ep,p_R])=\dim_H(\overline{\ub}\cap[p_L,p_R])=\frac{\log 2}{m\log p_R}\] for any $\ep>0$. Thus, one could say that within an entropy interval $[p_L,p_R]$, $\overline{\ub}$ is ``thickest" near the right endpoint $p_R$. 
}
\end{remark}

\begin{proof}[Proof of Theorem \ref{main:4}]
Since $\ub\setminus\bb\subset [q_{KL}(M),{M+1}]$, by (\ref{e15}) we have
  $\ub\setminus\bb=\set{q_{KL}}\cup\bigcup(\ub\cap[p_L, p_R])$, where the union is pairwise disjoint and countable. Then
  \begin{equation}\label{e411}
  \dim_H(\ub\setminus\bb)=\dim_H \bigcup_{[p_L, p_R]}(\ub\cap[p_L, p_R])=\sup_{[p_L, p_R]}\dim_H(\ub\cap[p_L, p_R]).
  \end{equation}
  Here the supremum is taken over all entropy plateaus $[p_L, p_R]{\subset(q_{KL}(M), M+1]}$. 
	
	Assume first that $M=1$. Recall that for any entropy plateau $[p_L, p_R]\subseteq(q_{KL}{(1)}, 2]$ with $\al(p_L)=(a_1\ldots a_m)^\f$, it holds that $m\geq 3$. Furthermore, $m=3$ if and only if $[p_L, p_R]=[\la_*, \la^*]\approx [1.83928, 1.87135]$, where
  $\al(\la_*)=(110)^\f$ {and} $\al(\la^*)=111(001)^\f$.  {Observe that $q_{KL}(1)\approx 1.78723$}. By a direct calculation one can verify that for any $m\ge 4$ we have
  \begin{equation}\label{e412}
  \frac{\log 2}{m\log p_R}<\frac{\log 2}{4\log q_{KL}}<\frac{\log 2}{3\log \lambda^*}.
  \end{equation}
 Therefore, by (\ref{e411}), (\ref{e412}) and Theorem \ref{thm:41} it follows that
  \begin{equation*}
  \dim_H(\ub\setminus\bb)%&=\sup_{[p_L, p_R]}\dim_H(\ub\cap[p_L, p_R])\\
  =\dim_H(\ub\cap[\la_*, \la^*])
	=\frac{\log 2}{3\log \la^*}\approx 0.368699.
  \end{equation*}
{Finally, since $\al(\la^*)=111(001)^\f$,  $\la^*$ is the unique root in $(1,2]$ of the equation 
	\[
	1=\frac{1}{x}+\frac{1}{x^2}+\frac{1}{x^3}+\frac{1}{x^3(x^3-1)},
	\]
	or equivalently, $x^5-x^4-x^3-2x^2+x+1=0$.}
	
Consider next the case $M=2$. Then $m\geq 2$, with equality if and only if $[p_L,p_R]=[\ga_*,\ga^*]\approx [2.73205,2.77462]$, where $\al(\ga_*)=(21)^\f$ and $\al(\ga^*)=22(01)^\f$. For any entropy plateau $[p_L,p_R]$ with period $m\geq 3$, we have $m\log p_R\geq 3\log q_{KL}{(2)}\geq 3\log {2.5}>2\log 3>2\log \ga^*$, so
\[
\frac{\log 2}{m\log p_R}<\frac{\log 2}{2\log\ga^*}.
\]
Hence, by \eqref{e411} and Theorem \ref{thm:41}, 
\[\dim_H(\ub\setminus\bb)=\dim_H(\ub\cap[\ga_*, \ga^*])=\frac{\log 2}{2\log\ga^*}\approx 0.339607.
\]
Furthermore, {since $\al(\ga^*)=22(01)^\f$, $\ga^*$ is the unique root in $(2,3)$ of the equation
\[
1=\frac{2}{x}+\frac{2}{x^2}+\frac{1}{x^2(x^2-1)},
\]
or equivalently,} $\ga^*$ is the unique root in $(2,3)$ of $x^4-2x^3-3x^2+2x+1=0$.
	
Finally, let $M\geq 3$. The leftmost entropy plateau with period $m=1$ is $[p_L,p_R]$, where 
\begin{alignat*}{3}
M &={2k+1} \quad &\Rightarrow &\quad \al(p_L)=(k+1)^\f \quad\mbox{and}\quad \al(p_R)=(k+2)k^\f,\\
M &={2k} \quad &\Rightarrow &\quad \al(p_L)=(k+1)^\f \quad\mbox{and}\quad \al(p_R)=(k+2)(k-1)^\f.
\end{alignat*}
Note that for this entropy plateau, $p_R=q_\star(M)$, where $q_\star(M)$ was defined in \eqref{e17}. Now consider an arbitrary entropy plateau $[p_L,p_R]$ with period $m$. If $m=1$, then $p_R\geq q_\star(M)$, so ${m\log p_R}\geq \log q_\star(M)$. And if $m\geq 2$, we have
\begin{align*}
m\log p_R&\geq 2\log q_{KL}{(M)}\geq 2\log\left(\frac{M+2}{2}\right)=\log(M^2+4M+4)-\log 4\\
&\geq \log(4M+4)-\log 4=\log(M+1)>\log q_\star(M).
\end{align*}
In both cases, we obtain
\[
\frac{\log 2}{m\log p_R}\leq \frac{\log 2}{\log q_\star(M)}.
\]
Hence, by \eqref{e411} and Theorem \ref{thm:41}, $\dim_H(\ub\setminus\bb)=\log 2/\log q_\star(M)$. This completes the proof.
\end{proof}

\section*{Acknowledgments}
{The authors thank the anonymous referee for many useful comments.} {The second author} was supported by the EPSRC grant EP/M001903/1. {The third author} was supported by NSFC No.~11401516.

%\bibliographystyle{abbrv}
%\bibliography{Fractal-Expansions}

\end{document}